\documentclass{article}
\usepackage{amsmath}
\usepackage{amssymb}

\newtheorem{theorem}{Theorem}

\newtheorem{corollary}[theorem]{Corollary}

\newtheorem{lemma}[theorem]{Lemma}

\newtheorem{proposition}[theorem]{Proposition}
\newtheorem{remark}[theorem]{Remark}

\newenvironment{proof}[1][Proof]{\noindent\textbf{#1.} }{\ \rule{0.5em}{0.5em}}
\input{tcilatex}
\begin{document}

\title{Semantical conditions for the definability of functions and relations}
\author{M. Campercholi, D. Vaggione}
\maketitle

\begin{abstract}
Let $\mathcal{L}\subseteq \mathcal{L}^{\prime }$ be first order languages,
let $R\in \mathcal{L}^{\prime }-\mathcal{L}$ be a relation symbol, and let $%
\mathcal{K}$ be a class of $\mathcal{L}^{\prime }$-structures. In this paper
we present semantical conditions equivalent to the existence of an $\mathcal{%
L}$-formula $\varphi \left( \vec{x}\right) $ such that $\mathcal{K}\vDash
\varphi (\vec{x})\leftrightarrow R(\vec{x})$, and $\varphi $ has a specific
syntactical form (e.g., quantifier free, positive and quantifier free,
existential horn, etc.). For each of these definability results for
relations we also present an analogous version for the definability of
functions. Several applications to natural definability questions in
universal algebra have been included; most notably definability of principal
congruences. The paper concludes with a look at term-interpolation in
classes of structures with the same techniques used for definability. Here
we obtain generalizations of two classical term-interpolation results:
Pixley's theorem \cite{Pixley} for quasiprimal algebras, and the
Baker-Pixley Theorem \cite{ba-pi} for finite algebras with a majority term.
\end{abstract}

\section*{Introduction}

Let $\mathcal{L}$ be a first order language and $\mathcal{K}$ a class of $%
\mathcal{L}$-structures. If $R\in \mathcal{L}$ is an $n$-ary relation
symbol, we say that a formula $\varphi (\vec{x})$ \emph{defines} $R$ \emph{in%
} $\mathcal{K}$ if%
\begin{equation*}
\mathcal{K}\vDash \varphi (\vec{x})\leftrightarrow R(\vec{x})\text{.}
\end{equation*}%
Let $f_{1},\ldots ,f_{m}\in \mathcal{L}$ be $n$-ary function symbols. Given
an $\mathcal{L}$-structure $\mathbf{A}$, let $\vec{f}^{\mathbf{A}%
}:A^{n}\rightarrow A^{m}$ be the function defined by 
\begin{equation*}
\vec{f}^{\mathbf{A}}\left( \vec{a}\right) =(f_{1}^{\mathbf{A}}(\vec{a}%
),\ldots ,f_{m}^{\mathbf{A}}(\vec{a}))\text{.}
\end{equation*}%
We say that a formula $\varphi (\vec{x},\vec{z})$ \emph{defines} $\vec{f}$ 
\emph{in} $\mathcal{K}$ if%
\begin{equation*}
\mathcal{K}\vDash \varphi (\vec{x},\vec{z})\leftrightarrow \vec{f}(\vec{x})=%
\vec{z}\text{.}
\end{equation*}%
Let $\mathcal{L}\subseteq \mathcal{L}^{\prime }$ be first order languages,
let $R\in \mathcal{L}^{\prime }-\mathcal{L}$ be an $n$-ary relation symbol
(resp. $f_{1},\ldots ,f_{m}\in \mathcal{L}^{\prime }-\mathcal{L}$ be $n$-ary
function symbols), and let $\mathcal{K}$ be a class of $\mathcal{L}^{\prime
} $-structures. Let $S$ be any of the following sets:

\begin{itemize}
\item $\{$finite conjunctions of atomic $\mathcal{L}$-formulas$\}$,

\item $\{$positive open $\mathcal{L}$-formulas$\}$,

\item $\{$open Horn $\mathcal{L}$-formulas$\}$,

\item $\{$open $\mathcal{L}$-formulas$\}$,

\item $\{$primitive positive $\mathcal{L}$-formulas$\}$,

\item $\{$existential positive $\mathcal{L}$-formulas$\}$,

\item $\{$existential Horn $\mathcal{L}$-formulas$\}$,

\item $\{$existential $\mathcal{L}$-formulas$\}$.
\end{itemize}

In this paper we give semantical conditions characterizing when $R$ (resp. $%
\vec{f}$) is definable in $\mathcal{K}$ by a formula of $S$. The results
obtained provide a natural and unified way to handle familiar questions on
definability of functions and relations in classes of structures. Being able
to look at a great variety of definability questions within the same
framework allows for a deeper understanding of the definability phenomena in
general. Evidence of this is our finding of several new results on
definability of principal congruences, and the generalizations of the
Baker-Pixley Theorem \cite{ba-pi} and of the Pixley Theorem \cite{Pixley}.
The applications throughout the paper provide a good sample of how the
results are put to work, in some cases providing direct proofs of known
facts and in others discovering new theorems.

In Section 2 we study the definability of functions and relations by
(positive) open formulas. We give some immediate applications to
definability of relative principal congruences in quasivarieties
(Proposition \ref{aplicacion a dpc} and Corollary \ref{aplicacion a dpc 1}).
In Section 3 we study the definability by open Horn formulas. In Section 4
we focus on the definability by conjunctions of atomic formulas. We give
some consequences on definability of principal congruences and the
Fraser-Horn property (Proposition \ref{aplicacion a dpc 2} and Corollary \ref%
{aplicacion a dpc 3}). We also apply the characterizations to give new
natural proofs of two results on the translation of (positive) open formulas
to conjunctions of equations (Proposition \ref{traduccion} and Corollary \ref%
{traduccion para simples}). In Section 5 we address definability by
existential formulas. Subsection 5.1 is concerned with primitive positive
functions. As an application we characterize primitive positive functions in
Stone algebras and De Morgan algebras. Section 6 is devoted to term
interpolation of functions in classes. First we apply the previous results
to characterize when a function is term valued by cases in a class (Theorems %
\ref{term valued with open cases} and \ref{term valued with positive cases}%
). We use this to give generalizations of Pixley's theorem \cite{Pixley}
characterizing quasiprimal algebras as those finite algebras in which every
function preserving the inner isomorphisms is a term function (Theorems \ref%
{generalizacion imp no trivial de pixley} and \ref{Pixley}). We conclude the
section giving two generalizations of the Baker-Pixley Theorem \cite{ba-pi}
on the existence of terms representing functions in finite algebras with a
majority term (Theorems \ref{Baker-Pixley para clases caso localmente
finitas} and \ref{Baker-Pixley para aritmeticas}).

Even though most results in the paper are true in more general contexts (via
the same ideas), we have preferred to write the results in a more concise
manner. The intention is to provide the non-specialist a more accessible
presentation, with the hope that he/she can find further natural
applications in universal algebra.

\section{Notation}

As usual, $\mathbb{I}(\mathcal{K})$, $\mathbb{S}(\mathcal{K})$, $\mathbb{P}(%
\mathcal{K})$ and $\mathbb{P}_{u}(\mathcal{K})$ denote the classes of
isomorphic images, substructures, direct products and ultraproducts of
elements of $\mathcal{K}$. We write $\mathbb{P}_{\mathrm{fin}}(\mathcal{K})$
to denote the class $\{\mathbf{A}_{1}\times \ldots \times \mathbf{A}%
_{n}:n\geq 1$ and each $\mathbf{A}_{i}\in \mathcal{K}\}$. For a class of
algebras $\mathcal{K}$ let $\mathbb{Q}(\mathcal{K})$ (resp. $\mathbb{V}(%
\mathcal{K})$) denote the \textit{quasivariety }(resp. \textit{variety}) 
\textit{generated by }$\mathcal{K}$. If $\mathcal{L}\subseteq \mathcal{L}%
^{\prime }$ are first order languages, for an $\mathcal{L}^{\prime }$%
-structure $\mathbf{A}$ we use $\mathbf{A}_{\mathcal{L}}$ to denote the
reduct of $\mathbf{A}$ to the language $\mathcal{L}$. If $\mathbf{A},\mathbf{%
B}$ are $\mathcal{L}$-structures, we write $\mathbf{A}\leq \mathbf{B}$ to
express that $\mathbf{A}$ is a substructure of $\mathbf{B}$.

Let $S_{1},\ldots ,S_{k}$ be non-empty sets, let $n\in N$. For $i=1,\ldots
,k $, let $f_{i}:S_{i}^{n}\rightarrow S_{i}$ be an $n$-ary operation on $%
S_{i}$. We use $f_{1}\times \ldots \times f_{k}$ to denote the function $%
f_{1}\times \ldots \times f_{k}:(S_{1}\times \ldots \times
S_{k})^{n}\rightarrow S_{1}\times \ldots \times S_{k}$ given by%
\begin{equation*}
f_{1}\times \ldots \times f_{k}((a_{1}^{1},\ldots ,a_{k}^{1}),\ldots
,(a_{1}^{n},\ldots ,a_{k}^{n}))=(f_{1}(a_{1}^{1},\ldots ,a_{1}^{n}),\ldots
,f_{k}(a_{k}^{1},\ldots ,a_{k}^{n}))\text{.}
\end{equation*}%
Also, if $R_{i}\subseteq S_{i}^{n}$ are $n$-ary relations on $S_{i}$, then
we write $R_{1}\times \ldots \times R_{k}$ to denote the $n$-ary relation
given by%
\begin{equation*}
R_{1}\times \ldots \times R_{k}=\{(a_{1}^{1},\ldots ,a_{k}^{1}),\ldots
,(a_{1}^{n},\ldots ,a_{k}^{n})):(a_{i}^{1},\ldots ,a_{i}^{n})\in R_{i},\ \
i=1,\ldots ,k\}\text{.}
\end{equation*}%
We observe that if $\mathbf{S}_{1},\ldots ,\mathbf{S}_{k}$ are $\mathcal{L}$%
-structures and $f\in \mathcal{L}$ is an $n$-ary operation symbol, then $f^{%
\mathbf{S}_{1}}\times \ldots \times f^{\mathbf{S}_{k}}=f^{\mathbf{S}%
_{1}\times \ldots \times \mathbf{S}_{k}}$. Also, if $R\in \mathcal{L}$ is an 
$n$-ary relation symbol, then $R^{\mathbf{S}_{1}}\times \ldots \times R^{%
\mathbf{S}_{k}}=R^{\mathbf{S}_{1}\times \ldots \times \mathbf{S}_{k}}$.

For a quasivariety $\mathcal{Q}$ and $\mathbf{A}\in \mathcal{Q}$, we use $%
\mathrm{Con}_{\mathcal{Q}}(\mathbf{A})$ to denote the lattice of relative
congruences of $\mathbf{A}$. If $a,b\in A$, with $\mathbf{A}\in \mathcal{Q}$%
, let $\theta _{\mathcal{Q}}^{\mathbf{A}}(a,b)$ denote the \emph{relative
principal congruence} generated by $(a,b)$. When $\mathcal{Q}$ is a variety
we drop the subscript and just write $\theta ^{\mathbf{A}}(a,b)$. The
quasivariety $\mathcal{Q}$ has \emph{definable relative principal
congruences }if there exists a first order formula $\varphi (x,y,z,w)$ in
the language of $\mathcal{Q}$ such that%
\begin{equation*}
\theta _{\mathcal{Q}}^{\mathbf{A}}(a,b)=\{(c,d):\mathcal{Q}\vDash \varphi
(a,b,c,d)\}
\end{equation*}%
for any $a,b\in A$, $\mathbf{A}\in \mathcal{Q}$. The quasivariety $\mathcal{Q%
}$ has the \emph{relative congruence extension property }if for every $%
\mathbf{A}\leq \mathbf{B}\in \mathcal{Q}$ and $\theta \in \mathrm{Con}_{%
\mathcal{Q}}(\mathbf{A})$ there is a $\delta \in \mathrm{Con}_{\mathcal{Q}}(%
\mathbf{B})$ such that $\theta =\delta \cap A^{2}$.

Let $\mathcal{Q}_{RFSI}$ (resp. $\mathcal{Q}_{RS}$) denote the class of
relative finitely subdirectly irreducible (resp. simple) members of $%
\mathcal{Q}$. When $\mathcal{Q}$ is a variety, we write $\mathcal{Q}_{FSI}$
in place of $\mathcal{Q}_{RFSI}$, $\mathrm{Con}(\mathbf{A})$ in place of $%
\mathrm{Con}_{\mathcal{Q}}(\mathbf{A})$, etc.

Let%
\begin{equation*}
\begin{array}{l}
\mathrm{At}(\mathcal{L})=\{\text{atomic}\mathrm{\ }\mathcal{L}\text{-formulas%
}\}\text{,} \\ 
\mathrm{\pm At}(\mathcal{L})=\mathrm{At}(\mathcal{L})\cup \{\lnot \alpha
:\alpha \in \mathrm{At}(\mathcal{L})\}\text{,} \\ 
\mathrm{Op}(\mathcal{L})=\{\varphi :\varphi \ \text{is\ an\ open}\mathrm{\ }%
\mathcal{L}\text{-formula}\}\text{,} \\ 
\mathrm{OpHorn}(\mathcal{L})=\{\varphi :\varphi \ \text{is\ an\ open\ Horn}%
\mathrm{\ }\mathcal{L}\text{-formula}\}\text{.}%
\end{array}%
\end{equation*}%
If $S$ is a set of formulas, we define%
\begin{equation*}
\begin{array}{l}
\left[ \bigwedge S\right] =\{\varphi _{1}\wedge \ldots \wedge \varphi
_{n}:\varphi _{1},\ldots ,\varphi _{n}\in S,\ n\geq 1\}\text{,} \\ 
\left[ \bigvee S\right] =\{\varphi _{1}\vee \ldots \vee \varphi _{n}:\varphi
_{1},\ldots ,\varphi _{n}\in S,\ n\geq 1\}\text{,} \\ 
\left[ \forall S\right] =\{\forall x_{1}\ldots \forall x_{n}\varphi :\varphi
\in S,\ n\geq 0\}\text{,} \\ 
\left[ \exists S\right] =\{\exists x_{1}\ldots \exists x_{n}\varphi :\varphi
\in S,\ n\geq 0\}\text{.}%
\end{array}%
\end{equation*}

\section{Definability by (positive) open formulas}

\begin{theorem}
\label{(positive) open para relaciones}Let $\mathcal{L}\subseteq \mathcal{L}%
^{\prime }$ be first order languages and let $R\in \mathcal{L}^{\prime }-%
\mathcal{L}$ be an $n$-ary relation symbol. For a class $\mathcal{K}$ of $%
\mathcal{L}^{\prime }$-structures, the following are equivalent:

\begin{enumerate}
\item[(1)] There is a formula in $\mathrm{Op}(\mathcal{L})$ (resp. $\left[
\bigvee \bigwedge \mathrm{At}(\mathcal{L})\right] $) which defines $R$ in $%
\mathcal{K}$.

\item[(2)] For all $\mathbf{A},\mathbf{B}\in \mathbb{P}_{u}(\mathcal{K})$,
all $\mathbf{A}_{0}\leq \mathbf{A}_{\mathcal{L}}$, $\mathbf{B}_{0}\leq 
\mathbf{B}_{\mathcal{L}}$, all isomorphisms (resp. homomorphisms) $\sigma :%
\mathbf{A}_{0}\rightarrow \mathbf{B}_{0}$, and all $a_{1},\ldots ,a_{n}\in
A_{0}$, we have that $(a_{1},\ldots ,a_{n})\in R^{\mathbf{A}}$ implies $%
(\sigma (a_{1}),\ldots ,\sigma (a_{n}))\in R^{\mathbf{B}}$.
\end{enumerate}

\noindent Moreover, if $\mathcal{K}_{\mathcal{L}}$ has finitely many
isomorphism types of $n$-generated substructures and each one is finite,
then we can remove the operator $\mathbb{P}_{u}$ from (2).
\end{theorem}

\begin{proof}
(1)$\Rightarrow $(2). Suppose that $\varphi (\vec{x})\in \mathrm{Op}(%
\mathcal{L})$ (resp. $\varphi (\vec{x})\in \left[ \bigvee \bigwedge \mathrm{%
At}(\mathcal{L})\right] $) defines $R$ in $\mathcal{K}$. Note that $\varphi (%
\vec{x})$ defines $R$ in $\mathbb{P}_{u}(\mathcal{K})$. Suppose that $%
\mathbf{A},\mathbf{B}\in \mathbb{P}_{u}(\mathcal{K})$, $\mathbf{A}_{0}\leq 
\mathbf{A}_{\mathcal{L}}$, $\mathbf{B}_{0}\leq \mathbf{B}_{\mathcal{L}}$, $%
\sigma :\mathbf{A}_{0}\rightarrow \mathbf{B}_{0}$ is an isomorphism (resp. a
homomorphism), and fix $a_{1},\ldots ,a_{n}\in A_{0}$ such that $%
(a_{1},\ldots ,a_{n})\in R^{\mathbf{A}}$. Since%
\begin{equation*}
\mathbf{A}\vDash \varphi (\vec{a})
\end{equation*}%
and $\varphi (\vec{x})\in \mathrm{Op}(\mathcal{L})$, we have that%
\begin{equation*}
\mathbf{A}_{0}\vDash \varphi (\vec{a})\text{.}
\end{equation*}%
Since $\sigma $ is an isomorphism (resp. $\sigma $ is a homomorphism and $%
\varphi (\vec{x})\in \left[ \bigvee \bigwedge \mathrm{At}(\mathcal{L})\right]
$), we have that%
\begin{equation*}
\mathbf{B}_{0}\vDash \varphi (\sigma (a_{1}),\ldots ,\sigma (a_{n}))\text{.}
\end{equation*}%
As $\varphi (\vec{x})\in \mathrm{Op}(\mathcal{L})$, it follows that%
\begin{equation*}
\mathbf{B}\vDash \varphi (\sigma (a_{1}),\ldots ,\sigma (a_{n}))\text{,}
\end{equation*}%
and thus $(\sigma (a_{1}),\ldots ,\sigma (a_{n}))\in R^{\mathbf{B}}$.

(2)$\Rightarrow $(1). Let $\mathbf{A}\in \mathbb{P}_{u}(\mathcal{K})$ and $%
\vec{a}\in R^{\mathbf{A}}$. Define%
\begin{equation*}
\Delta ^{\vec{a},\mathbf{A}}=\{\alpha (\vec{x}):\alpha \in \mathrm{\pm At}(%
\mathcal{L})\text{, }\mathbf{A}\vDash \alpha (\vec{a})\}
\end{equation*}%
(resp. $\Delta ^{\vec{a},\mathbf{A}}=\{\alpha (\vec{x}):\alpha \in \mathrm{At%
}(\mathcal{L})$ and $\mathbf{A}\vDash \alpha (\vec{a})\}$). Take $\mathbf{B}%
\in \mathbb{P}_{u}(\mathcal{K})$ and $\vec{b}\in B^{n}$ such that $\mathbf{B}%
\vDash \Delta ^{\vec{a},\mathbf{A}}(\vec{b})$. Let%
\begin{equation*}
\mathbf{A}_{0}=\text{the substructure of }\mathbf{A}_{\mathcal{L}}\text{
generated by }a_{1},\ldots ,a_{n}\text{,}
\end{equation*}%
\begin{equation*}
\mathbf{B}_{0}=\text{the substructure of }\mathbf{B}_{\mathcal{L}}\text{
generated by }b_{1},\ldots ,b_{n}\text{.}
\end{equation*}%
Since $\mathbf{B}\vDash \Delta ^{\vec{a},\mathbf{A}}(\vec{b})$, we have that%
\begin{equation*}
a_{1}\mapsto b_{1},\dots ,a_{n}\mapsto b_{n},\ \ f_{1}^{\mathbf{A}}(\vec{a}%
)\mapsto c_{1},\dots ,f_{m}^{\mathbf{A}}(\vec{a})\mapsto c_{m}
\end{equation*}%
extends to an isomorphism (resp. homomorphism) from $\mathbf{A}_{0}$ onto $%
\mathbf{B}_{0}$, which by (2) says that $\vec{b}\in R^{\mathbf{B}}$. So we
have proved that%
\begin{equation*}
\mathbb{P}_{u}(\mathcal{K})\vDash \left( \dbigwedge\limits_{\alpha \in
\Delta ^{\vec{a},\mathbf{A}}}\alpha (\vec{x})\right) \rightarrow R\left( 
\vec{x}\right) \text{.}
\end{equation*}%
By compactness, there is a finite subset $\Delta _{0}^{\vec{a},\mathbf{A}%
}\subseteq \Delta ^{\vec{a},\mathbf{A}}$ such that%
\begin{equation*}
\mathbb{P}_{u}(\mathcal{K})\vDash \left( \dbigwedge\limits_{\alpha \in
\Delta _{0}^{\vec{a},\mathbf{A}}}\alpha (\vec{x})\right) \rightarrow R(\vec{x%
})\text{.}
\end{equation*}%
Next note that%
\begin{equation*}
\mathbb{P}_{u}(\mathcal{K})\vDash \left( \dbigvee\limits_{\mathbf{A}\in 
\mathbb{P}_{u}(\mathcal{K})\text{, }\vec{a}\in R^{\mathbf{A}%
}}\dbigwedge\limits_{\alpha \in \Delta _{0}^{\vec{a},\mathbf{A}}}\alpha (%
\vec{x})\right) \leftrightarrow R(\vec{x})\text{,}
\end{equation*}%
which by compactness says that%
\begin{equation*}
\mathbb{P}_{u}(\mathcal{K})\vDash \left(
\dbigvee\limits_{i=1}^{k}\dbigwedge\limits_{\alpha \in \Delta _{0}^{\vec{a}%
_{i},\mathbf{A}_{i}}}\alpha (\vec{x})\right) \leftrightarrow R(\vec{x})
\end{equation*}%
for some $\mathbf{A}_{1},\ldots ,\mathbf{A}_{k}\in \mathbb{P}_{u}(\mathcal{K}%
)$, $\vec{a}_{1}\in R^{\mathbf{A}_{1}},\ldots ,\vec{a}_{k}\in R^{\mathbf{A}%
_{k}}$.

Now we prove the moreover part. Suppose $\mathcal{K}_{\mathcal{L}}$ has
finitely many isomorphism types of $n$-generated substructures and each one
is finite. Thus, there is a finite list of atomic $\mathcal{L}$-formulas $%
\alpha _{1}(\vec{x}),\ldots ,\alpha _{k}(\vec{x})$ such that for every
atomic $\mathcal{L}$-formula $\alpha (\vec{x})$, there is $j\in \{1,\dots
,k\}$ satisfying $\mathcal{K}\vDash \alpha (\vec{x})\leftrightarrow \alpha
_{j}(\vec{x})$. Assume (2) holds without the ultraproduct operator, we prove
(1). By considerations similar to the above we have that%
\begin{equation*}
\mathcal{K}\vDash \left( \dbigvee\limits_{\mathbf{A}\in \mathcal{K}\text{, }%
\vec{a}\in R^{\mathbf{A}}}\dbigwedge\limits_{\alpha \in \Delta ^{\vec{a},%
\mathbf{A}}}\alpha (\vec{x})\right) \leftrightarrow R(\vec{x})\text{.}
\end{equation*}%
Since each $\Delta ^{\vec{a},\mathbf{A}}$ can be supposed to be included in%
\begin{equation*}
\{\alpha _{1}(\vec{x}),\ldots ,\alpha _{k}(\vec{x})\}\cup \{\lnot \alpha
_{1}(\vec{x}),\ldots ,\lnot \alpha _{k}(\vec{x})\}
\end{equation*}%
(resp. $\{\alpha _{1}(\vec{x}),\ldots ,\alpha _{k}(\vec{x})\}$), we have
(after removing redundancies) that,%
\begin{equation*}
\dbigvee\limits_{\mathbf{A}\in \mathcal{K}\text{, }\vec{a}\in R^{\mathbf{A}%
}}\dbigwedge\limits_{\alpha \in \Delta ^{\vec{a},\mathbf{A}}}\alpha (\vec{x})
\end{equation*}%
is a first order formula.

Finally, note that the remaining implication is already taken care of by (1)$%
\Rightarrow $(2) above.
\end{proof}

Here is a direct consequence of the above theorem.

\begin{corollary}
\label{(positive) open locally finite para relaciones}Let $\mathcal{K}$ be
any class of $\mathcal{L}$-algebras contained in a locally finite variety.
Suppose $\mathbf{A}\rightarrow R^{\mathbf{A}}$ is a map which assigns to
each $\mathbf{A}\in \mathcal{K}$ an $n$-ary relation $R^{\mathbf{A}%
}\subseteq A^{n}$. The following are equivalent:

\begin{enumerate}
\item[(1)] There is a formula in $\mathrm{Op}(\mathcal{L})$ (resp. $\left[
\bigvee \bigwedge \mathrm{At}(\mathcal{L})\right] $) which defines $R$ in $%
\mathcal{K}$.

\item[(2)] For all $\mathbf{A},\mathbf{B}\in \mathcal{K}$, all $\mathbf{A}%
_{0}\leq \mathbf{A}_{\mathcal{L}}$, $\mathbf{B}_{0}\leq \mathbf{B}_{\mathcal{%
L}}$, all isomorphisms (resp. homomorphisms) $\sigma :\mathbf{A}%
_{0}\rightarrow \mathbf{B}_{0}$, and all $a_{1},\ldots ,a_{n}\in A_{0}$ we
have that $(a_{1},\ldots ,a_{n})\in R^{\mathbf{A}}$ implies $(\sigma
(a_{1}),\ldots ,\sigma (a_{n}))\in R^{\mathbf{B}}$.
\end{enumerate}
\end{corollary}

\begin{proof}
Apply the moreover part of Theorem \ref{(positive) open para relaciones} to
the class $\{(\mathbf{A},R^{\mathbf{A}}):\mathbf{A}\in \mathcal{K}\}$.
\end{proof}

As we shall see next it is easy to derive the functional version of Theorem %
\ref{(positive) open para relaciones}.

\begin{theorem}
\label{(positive) open}Let $\mathcal{L}\subseteq \mathcal{L}^{\prime }$ be
first order languages and let $f_{1},\ldots ,f_{m}\in \mathcal{L}^{\prime }-%
\mathcal{L}$ be $n$-ary function symbols. For a class $\mathcal{K}$ of $%
\mathcal{L}^{\prime }$-structures, the following are equivalent:

\begin{enumerate}
\item[(1)] There is a formula in $\mathrm{Op}(\mathcal{L})$ (resp. $\left[
\bigvee \bigwedge \mathrm{At}(\mathcal{L})\right] $) which defines $\vec{f}$
in $\mathcal{K}$.

\item[(2)] For all $\mathbf{A},\mathbf{B}\in \mathbb{P}_{u}(\mathcal{K})$,
all $\mathbf{A}_{0}\leq \mathbf{A}_{\mathcal{L}}$, $\mathbf{B}_{0}\leq 
\mathbf{B}_{\mathcal{L}}$, all isomorphisms (resp. homomorphisms) $\sigma :%
\mathbf{A}_{0}\rightarrow \mathbf{B}_{0}$, and all $a_{1},\ldots ,a_{n}\in
A_{0}$ such that\newline
$f_{1}^{\mathbf{A}}(\vec{a}),\ldots ,f_{m}^{\mathbf{A}}(\vec{a})\in A_{0}$,
we have $\sigma (f_{i}^{\mathbf{A}}(\vec{a}))=f_{i}^{\mathbf{B}}(\sigma
(a_{1}),\ldots ,\sigma (a_{n}))$ for all $i\in \{1,\ldots ,m\}$.
\end{enumerate}

\noindent Moreover, if $\mathcal{K}_{\mathcal{L}}$ has finitely many
isomorphism types of $(n+m)$-generated substructures and each one is finite,
then we can remove the operator $\mathbb{P}_{u}$ from (2).
\end{theorem}

\begin{proof}
The implication (1)$\Rightarrow $(2) is analogous to (1)$\Rightarrow $(2) of
Theorem \ref{(positive) open para relaciones}.

We prove (2)$\Rightarrow $(1). Let $R\notin \mathcal{L}^{\prime }$ be an $%
(n+m)$-ary relational symbol. For each $\mathbf{A}\in \mathcal{K}$ let%
\begin{equation*}
R^{\mathbf{A}}=\{\left( a_{1},\ldots ,a_{n},f_{1}^{\mathbf{A}}(\vec{a}%
),\ldots ,f_{m}^{\mathbf{A}}(\vec{a})\right) :a_{1},\ldots ,a_{n}\in A\}%
\text{.}
\end{equation*}%
Define the following class of $\mathcal{L}^{\prime }\cup \{R\}$-structures 
\begin{equation*}
\mathcal{K}_{R}=\{\left( \mathbf{A},R^{\mathbf{A}}\right) :\mathbf{A}\in 
\mathcal{K}\}\text{.}
\end{equation*}%
It is not hard to check that the isomorphism (resp. homomorphism) version of
(2) in Theorem \ref{(positive) open para relaciones} holds for the languages 
$\mathcal{L}\subseteq \mathcal{L}^{\prime }\cup \{R\}$, the relation symbol $%
R$ and the class $\mathcal{K}_{R}$. Thus there is $\varphi \left(
x_{1},\dots ,x_{n},z_{1},\dots ,z_{m}\right) \in \mathrm{Op}(\mathcal{L})$
(resp. $\left[ \bigvee \bigwedge \mathrm{At}(\mathcal{L})\right] $) such
that $\mathcal{K}_{R}\vDash \varphi \left( \vec{x},\vec{z}\right)
\leftrightarrow R\left( \vec{x},\vec{z}\right) $. Now as $\mathcal{K}%
_{R}\vDash \vec{f}(\vec{x})=\vec{z}\leftrightarrow R\left( \vec{x},\vec{z}%
\right) $, it immediately follows that $\varphi $ defines $\vec{f}$ in $%
\mathcal{K}_{R}$. Hence $\varphi $ defines $\vec{f}$ in $\mathcal{K}$.

Next we prove the moreover part. Suppose $\mathcal{K}_{\mathcal{L}}$ has
finitely many isomorphism types of $(n+m)$-generated substructures and each
one is finite. Assume (2) without $\mathbb{P}_{u}$. Let $\mathcal{K}_{R}$ be
as in the first part of the proof. Note that $\left( \mathcal{K}_{R}\right)
_{\mathcal{L}}=\mathcal{K}_{\mathcal{L}}$ and thus we can apply the moreover
part of Theorem \ref{(positive) open para relaciones} to obtain $\varphi \in 
\mathrm{Op}(\mathcal{L})$ (resp. $\left[ \bigvee \bigwedge \mathrm{At}(%
\mathcal{L})\right] $) defining $R$ in $\mathcal{K}_{R}$. Clearly $\varphi $
defines $\vec{f}$ in $\mathcal{K}$. The remaining implication is immediate
by (1)$\Rightarrow $(2).
\end{proof}

\begin{corollary}
\label{open para localy finite}Let $\mathcal{K}$ be any class of $\mathcal{L}
$-algebras contained in a locally finite variety. Suppose $\mathbf{A}%
\rightarrow f^{\mathbf{A}}$ is a map which assigns to each $\mathbf{A}\in 
\mathcal{K}$ an $n$-ary operation $f^{\mathbf{A}}:A^{n}\rightarrow A$. The
following are equivalent:

\begin{enumerate}
\item[(1)] There is a formula in $\mathrm{Op}(\mathcal{L})$ (resp. $\left[
\bigvee \bigwedge \mathrm{At}(\mathcal{L})\right] $) which defines $f$ in $%
\mathcal{K}$.

\item[(2)] For all $\mathbf{A},\mathbf{B}\in \mathcal{K}$, all $\mathbf{A}%
_{0}\leq \mathbf{A}_{\mathcal{L}}$, $\mathbf{B}_{0}\leq \mathbf{B}_{\mathcal{%
L}}$, all isomorphisms (resp. homomorphisms) $\sigma :\mathbf{A}%
_{0}\rightarrow \mathbf{B}_{0}$, and all $a_{1},\ldots ,a_{n}\in A_{0}$ such
that $f^{\mathbf{A}}(\vec{a})\in A_{0}$, we have $\sigma (f^{\mathbf{A}}(%
\vec{a}))=f^{\mathbf{B}}(\sigma (a_{1}),\ldots ,\sigma (a_{n}))$.
\end{enumerate}
\end{corollary}

\begin{proof}
Apply the moreover part of Theorem \ref{(positive) open} to the class $\{(%
\mathbf{A},f^{\mathbf{A}}):\mathbf{A}\in \mathcal{K}\}$.
\end{proof}

Let $\mathbf{3}=(\{0,1/2,1\},\max ,\min ,^{\ast },0,1)$, where $0^{\ast }=1$
and $(1/2)^{\ast }=1^{\ast }=0$. Of course, $\mathbf{3}$ is the
three-element Stone algebra (see \cite{ba-dw}). Note that the only non
trivial homomorphism between subalgebras of $\mathbf{3}$ is the map $^{\ast
\ast }:\{0,1/2,1\}\rightarrow \{0,1\}$. Thus the above corollary applied to
the class $\mathcal{K}=\{\mathbf{3}\}$ says that a function $%
f:\{0,1/2,1\}^{n}\rightarrow \{0,1/2,1\}$ is definable in $\mathbf{3}$ by a
positive open formula in the language of $\mathbf{3}$ iff $f(x_{1},\ldots
,x_{n})^{\ast \ast }=f(x_{1}^{\ast \ast },\ldots ,x_{n}^{\ast \ast })$, for
any $x_{1},\ldots ,x_{n}\in \{0,1/2,1\}$.

\subsection*{Applications to definable principal congruences}

We apply the above results to give natural proofs of two results on
definability of relative principal congruences in quasivarieties.

\begin{proposition}
\label{aplicacion a dpc}Let $\mathcal{Q}$ be a quasivariety with definable
relative principal congruences and let $\mathcal{L}$ be the language of $%
\mathcal{Q}$. The following are equivalent:

\begin{enumerate}
\item[(1)] There is a formula in $\left[ \forall \mathrm{Op}(\mathcal{L})%
\right] $ defining relative principal congruences in $\mathcal{Q}$.

\item[(2)] There is a formula in $\left[ \bigvee \bigwedge \mathrm{At}(%
\mathcal{L})\right] $ defining relative principal congruences in $\mathcal{Q}
$.

\item[(3)] $\mathcal{Q}$ has the relative congruence extension property.
\end{enumerate}
\end{proposition}

\begin{proof}
(1)$\Rightarrow $(3). We need the following fact proved in \cite{bl-pi}.

\begin{itemize}
\item A quasivariety $\mathcal{Q}$ has the relative congruence extension
property if for every $\mathbf{A},\mathbf{B}\in \mathcal{Q}$ with $\mathbf{A}%
\leq \mathbf{B}$ and for all $a,b\in A$ we have that $\theta _{\mathcal{Q}}^{%
\mathbf{A}}\left( a,b\right) =\theta _{\mathcal{Q}}^{\mathbf{B}}\left(
a,b\right) \cap A^{2}$.
\end{itemize}

\noindent Note that it is always the case that $\theta _{\mathcal{Q}}^{%
\mathbf{A}}\left( a,b\right) \subseteq \theta _{\mathcal{Q}}^{\mathbf{B}%
}\left( a,b\right) \cap A^{2}$. So, as formulas in $\left[ \forall \mathrm{Op%
}(\mathcal{L})\right] $ are preserved by subalgebras, (1) implies that $%
\theta _{\mathcal{Q}}^{\mathbf{A}}\left( a,b\right) =\theta _{\mathcal{Q}}^{%
\mathbf{B}}\left( a,b\right) \cap A^{2}$ for every $\mathbf{A}\leq \mathbf{B}%
\in \mathcal{Q}$. Thus the fact cited above yields (3).

(3)$\Rightarrow $(2). We use the following well known fact.

\begin{enumerate}
\item[(i)] For all $\mathbf{A},\mathbf{B}\in \mathcal{Q}$ and all
homomorphisms $\sigma :\mathbf{A}\rightarrow \mathbf{B}$ we have that $%
(a,b)\in \theta _{\mathcal{Q}}^{\mathbf{A}}\left( c,d\right) $ implies $%
(\sigma (a),\sigma (b))\in \theta _{\mathcal{Q}}^{\mathbf{B}}\left( \sigma
\left( c\right) ,\sigma \left( d\right) \right) $.
\end{enumerate}

\noindent For each $\mathbf{A}\in \mathcal{Q}$, let%
\begin{equation*}
R^{\mathbf{A}}=\{(a,b,c,d):(a,b)\in \theta _{\mathcal{Q}}^{\mathbf{A}}\left(
c,d\right) \}\text{,}
\end{equation*}%
and define%
\begin{equation*}
\mathcal{K}=\{\left( \mathbf{A},R^{\mathbf{A}}\right) :\mathbf{A}\in 
\mathcal{Q}\}\text{.}
\end{equation*}%
Since $\mathcal{Q}$ has definable relative principal congruences, $\mathcal{K%
}$ is a first order class and hence $\mathbb{P}_{u}(\mathcal{K})\subseteq 
\mathcal{K}$. Thus, in order to prove that (2) of Theorem \ref{(positive)
open para relaciones} holds we need to check that:

\begin{enumerate}
\item[(ii)] For all $\mathbf{A},\mathbf{B}\in \mathcal{K}$, all $\mathbf{A}%
_{0}\leq \mathbf{A}_{\mathcal{L}}$, $\mathbf{B}_{0}\leq \mathbf{B}_{\mathcal{%
L}}$, all homomorphisms $\sigma :\mathbf{A}_{0}\rightarrow \mathbf{B}_{0}$
and all $a,b,c,d\in A_{0}$, we have that $(a,b,c,d)\in R^{\mathbf{A}}$
implies $(\sigma (a),\sigma (b),\sigma (c),\sigma (d))\in R^{\mathbf{B}}$.
\end{enumerate}

\noindent Or, equivalently:

\begin{enumerate}
\item[(iii)] For all $\mathbf{A},\mathbf{B}\in \mathcal{Q}$, all $\mathbf{A}%
_{0}\leq \mathbf{A}_{\mathcal{L}}$, $\mathbf{B}_{0}\leq \mathbf{B}_{\mathcal{%
L}}$, all homomorphisms $\sigma :\mathbf{A}_{0}\rightarrow \mathbf{B}_{0}$
and all $a,b,c,d\in A_{0}$, we have that $(a,b)\in \theta _{\mathcal{Q}}^{%
\mathbf{A}}\left( c,d\right) $ implies $(\sigma (a),\sigma (b))\in \theta _{%
\mathcal{Q}}^{\mathbf{B}}(\sigma (c),\sigma (d))$.
\end{enumerate}

\noindent Since $\mathcal{V}$ has the congruence extension property, we can
replace in (iii) the occurrence of "$(a,b)\in \theta _{\mathcal{Q}}^{\mathbf{%
A}}\left( c,d\right) $" by "$(a,b)\in \theta _{\mathcal{Q}}^{\mathbf{A}%
_{0}}(c,d)$". Hence (iii) follows from (i).

(2)$\Rightarrow $(1). This is trivial.
\end{proof}

\begin{corollary}
\label{aplicacion a dpc 1}Let $\mathcal{Q}$ be a locally finite quasivariety
with the relative congruence extension property. Let $\mathcal{L}$ be the
language of $\mathcal{Q}$. Then there is a formula in $\left[ \bigvee
\bigwedge \mathrm{At}(\mathcal{L})\right] $ which defines the relative
principal congruences in $\mathcal{Q}$.
\end{corollary}

\begin{proof}
The proof is similar to (3)$\Rightarrow $(2) in the above proof, but
applying Corollary \ref{(positive) open locally finite para relaciones} in
place of Theorem \ref{(positive) open para relaciones}.
\end{proof}

The above corollary is proved in \cite{ba-be} for the case in which $%
\mathcal{Q}$ is a finitely generated variety.

\section{Definability by open Horn formulas}

\begin{theorem}
\label{open horn para relaciones}Let $\mathcal{L}\subseteq \mathcal{L}%
^{\prime }$ be first order languages and let $R\in \mathcal{L}^{\prime }-%
\mathcal{L}$ be an $n$-ary relation symbol. Let $\mathcal{K}$ be any class
of $\mathcal{L}^{\prime }$-structures. The following are equivalent:

\begin{enumerate}
\item[(1)] There is a formula in $\mathrm{OpHorn}(\mathcal{L})$ which
defines $R$ in $\mathcal{K}$.

\item[(2)] For all $\mathbf{A},\mathbf{B}\in \mathbb{P}_{u}\mathbb{P}_{%
\mathrm{fin}}(\mathcal{K})$, all $\mathbf{A}_{0}\leq \mathbf{A}_{\mathcal{L}%
} $, $\mathbf{B}_{0}\leq \mathbf{B}_{\mathcal{L}}$, all isomorphisms $\sigma
:\mathbf{A}_{0}\rightarrow \mathbf{B}_{0}$, and all $a_{1},\ldots ,a_{n}\in
A_{0}$, we have that $(a_{1},\ldots ,a_{n})\in R^{\mathbf{A}}$ implies $%
(\sigma (a_{1}),\ldots ,\sigma (a_{n}))\in R^{\mathbf{B}}$.
\end{enumerate}

\noindent Moreover, if $\mathcal{K}_{\mathcal{L}}$ has finitely many
isomorphism types of $n$-generated substructures and each one is finite,
then we can remove the operator $\mathbb{P}_{u}$ from (2).
\end{theorem}

\begin{proof}
(1)$\Rightarrow $(2). Note that if $\varphi \left( \vec{x}\right) \in 
\mathrm{OpHorn}(\mathcal{L})$ defines $R$ in $\mathcal{K}$, then $\varphi $
defines $R$ in $\mathbb{P}_{u}\mathbb{P}_{\mathrm{fin}}(\mathcal{K})$ as
well. Now we can repeat the argument of (1)$\Rightarrow $(2) in the proof of
Theorem \ref{(positive) open para relaciones}.

(2)$\Rightarrow $(1). Applying Theorem \ref{(positive) open para relaciones}
to the class $\mathbb{P}_{\mathrm{fin}}(\mathcal{K})$ we have that there is
an open $\mathcal{L}$-formula $\varphi $ which defines $R$ in $\mathbb{P}_{%
\mathrm{fin}}(\mathcal{K})$. W.l.o.g. we can suppose that%
\begin{equation*}
\varphi =\bigwedge_{j=1}^{r}\left( \pi _{j}\rightarrow
\bigvee_{i=1}^{k_{j}}\alpha _{i}^{j}\right) \wedge
\bigwedge_{j=1}^{l}\bigvee_{i=1}^{u_{j}}\lnot \beta _{i}^{j}
\end{equation*}%
with $r,k_{j},u_{j}\geq 1$, $l\geq 0$, each $\pi _{j}$ in $\left[ \bigwedge 
\mathrm{At}(\mathcal{L})\right] $ and the formulas $\alpha _{i}^{j},\beta
_{i}^{j}$ in $\mathrm{At}(\mathcal{L})$. Note that for $\mathbf{A}\in 
\mathbb{P}_{\mathrm{fin}}(\mathcal{K})$ and $\vec{a}\in R^{\mathbf{A}}$ we
have that%
\begin{equation*}
\mathbf{A}\vDash \left( \bigwedge_{j=1}^{l}\bigvee_{i=1}^{u_{j}}\lnot \beta
_{i}^{j}\right) (\vec{a})\text{.}
\end{equation*}%
Let%
\begin{equation*}
S=\{(s_{1},\ldots ,s_{r}):1\leq s_{j}\leq k_{j},\text{ \ }j=1,\ldots ,r\}%
\text{.}
\end{equation*}%
We claim that there is $s\in S$ such that $\bigwedge_{j=1}^{r}\left( \pi
_{j}\rightarrow \alpha _{s_{j}}^{j}\right) \wedge
\bigwedge_{j=1}^{l}\bigvee_{i=1}^{u_{j}}\lnot \beta _{i}^{j}$ defines $R$ in 
$\mathcal{K}$. For the sake of contradiction assume that this is not the
case. Then for each $s\in S$ there are $\mathbf{A}_{s}\in \mathcal{K}$, $%
\vec{a}_{s}=(a_{s1},\dots ,a_{sn})\in R^{\mathbf{A}_{s}}$ and $j_{s}$ such
that%
\begin{equation*}
\mathbf{A}_{s}\vDash \lnot \left( \pi _{j_{s}}\rightarrow \alpha
_{s_{j_{s}}}^{j_{s}}\right) (\vec{a}_{s})\text{,}
\end{equation*}%
or equivalently%
\begin{equation}
\mathbf{A}_{s}\vDash \pi _{j_{s}}(\vec{a}_{s})\wedge \lnot \alpha
_{s_{j_{s}}}^{j_{s}}(\vec{a}_{s})\text{.}  \tag{i}  \label{ec prueba horn}
\end{equation}%
Let $p_{1}=(a_{s1})_{s\in S},\dots ,p_{n}=(a_{sn})_{s\in S}$. Note that $%
\vec{p}\in R^{\Pi _{S}\mathbf{A}_{s}}$, and as $\varphi $ defines $R$ in $%
\mathbb{P}_{\mathrm{fin}}(\mathcal{K})$, we have%
\begin{equation*}
\prod_{s\in S}\mathbf{A}_{s}\vDash \varphi (\vec{p})\text{.}
\end{equation*}%
This implies%
\begin{equation*}
\prod_{s\in S}\mathbf{A}_{s}\vDash \bigwedge_{j=1}^{r}\left( \pi
_{j}\rightarrow \bigvee_{i=1}^{k_{j}}\alpha _{i}^{j}\right) (\vec{p})\text{.}
\end{equation*}%
By this and (\ref{ec prueba horn}) we have%
\begin{equation*}
\prod_{s\in S}\mathbf{A}_{s}\vDash \left( \bigwedge_{j=1}^{r}\pi _{j}\right)
(\vec{p})\text{.}
\end{equation*}%
Hence for each $j\in \{1,\ldots ,r\}$ there is an $s_{j}$ such that%
\begin{equation*}
\prod_{s\in S}\mathbf{A}_{s}\vDash \alpha _{s_{j}}^{j}(\vec{p})\text{.}
\end{equation*}%
Let $s=(s_{1},\ldots ,s_{r})$. Then we have that%
\begin{equation*}
\mathbf{A}_{s}\vDash \alpha _{s_{j_{s}}}^{j_{s}}(\vec{a}_{s})\text{,}
\end{equation*}%
which contradicts (\ref{ec prueba horn}).

To prove the moreover part, suppose $\mathcal{K}_{\mathcal{L}}$ has finitely
many isomorphism types of $n$-generated substructures and each one is
finite. Assume (2) without the ultraproduct operator. We prove (1). Note
that our hypothesis on $\mathcal{K}_{\mathcal{L}}$ implies that there are
atomic $\mathcal{L}$-formulas $\alpha _{1}(x_{1},\dots ,x_{n}),\ldots
,\alpha _{k}(x_{1},\dots ,x_{n})$ such that for every atomic $\mathcal{L}$%
-formula $\alpha (\vec{x})$ there is $j$ satisfying $\mathcal{K}\vDash
\alpha (\vec{x})\leftrightarrow \alpha _{j}(\vec{x})$. Since atomic formulas
are preserved by direct products and by direct factors, we have that for
every atomic $\mathcal{L}$-formula $\alpha (\vec{x})$, there is $j$ such
that $\mathbb{P}_{\mathrm{fin}}(\mathcal{K})\vDash \alpha (\vec{x}%
)\leftrightarrow \alpha _{j}(\vec{x})$. This implies that $\mathbb{P}_{%
\mathrm{fin}}(\mathcal{K})_{\mathcal{L}}$ has finitely many isomorphism
types of $n$-generated substructures and each one is finite. By Theorem \ref%
{(positive) open para relaciones} there is an open $\mathcal{L}$-formula $%
\varphi $ which defines $R$ in $\mathbb{P}_{\mathrm{fin}}(\mathcal{K})$. Now
we can proceed as in the first part of this proof.
\end{proof}

By a \emph{trivial }$\mathcal{L}$-structure we mean a structure $\mathbf{A}$
such that $A=\{a\}$ and $(a,\ldots ,a)\in R^{\mathbf{A}}$, for every $R\in 
\mathcal{L}$. Recall that a \emph{strict Horn formula} is a Horn formula
that has exactly one non-negated atomic formula in each of its clauses. Let
us write $\mathrm{OpStHorn}(\mathcal{L})$ for the set of open strict Horn $%
\mathcal{L}$-formulas.

\begin{remark}
Theorem \ref{open horn para relaciones} holds if we replace in (1) $\mathrm{%
OpHorn}(\mathcal{L})$ by $\mathrm{OpStHorn}(\mathcal{L})$ and add the
following requirement to (2):

\begin{description}
\item For all $\mathbf{A}\in \mathbb{P}_{u}(\mathcal{K})_{\mathcal{L}}$ with
a trivial substructure $\{a\}$, we have $(a,\ldots ,a)\in R^{\mathbf{A}}$.
\end{description}
\end{remark}

\begin{proof}
(1)$\Rightarrow $(2). Observe that formulas in $\mathrm{OpStHorn}(\mathcal{L}%
)$ are always satisfied in trivial structures.

To see (2)$\Rightarrow $(1) note that by Theorem \ref{open horn para
relaciones} there is a formula%
\begin{equation*}
\varphi (\vec{x})=\bigwedge_{j=1}^{r}\left( \pi _{j}\rightarrow \alpha
_{j}\right) \wedge \bigwedge_{j=1}^{l}\bigvee_{i=1}^{u_{j}}\lnot \beta
_{i}^{j}
\end{equation*}%
(with $r,u_{j}\geq 1$, $l\geq 0$, each $\pi _{j}$ in $\left[ \bigwedge 
\mathrm{At}(\mathcal{L})\right] $ and the formulas $\alpha _{j},\beta
_{i}^{j}$ in $\mathrm{At}(\mathcal{L})$) which defines $R$ in $\mathcal{K}$.
Note that $\varphi (\vec{x})$ also defines $R$ in $\mathbb{P}_{u}(\mathcal{K}%
)$. Assume $l\geq 1$ and suppose $\mathbf{A}\in \mathbb{P}_{u}(\mathcal{K})_{%
\mathcal{L}}$ has a trivial substructure $\{a\}$. Note that the additional
condition of (2) says that%
\begin{equation*}
\mathbf{A}\vDash \varphi (a,\ldots ,a)\text{,}
\end{equation*}%
which is absurd since $l\geq 1$. Thus we have proved that there is no
trivial substructure in $\mathbb{P}_{u}(\mathcal{K})_{\mathcal{L}}$. Hence%
\begin{equation*}
\mathbb{P}_{u}(\mathcal{K})\vDash \dbigvee\limits_{\alpha (z_{1})\in \mathrm{%
At}(\mathcal{L})}\lnot \alpha (z_{1})\text{,}
\end{equation*}%
which by compactness says that%
\begin{equation*}
\mathcal{K}\vDash \lnot \tilde{\alpha}_{1}(z_{1})\vee \ldots \vee \lnot 
\tilde{\alpha}_{k}(z_{1})
\end{equation*}%
for some atomic $\mathcal{L}$-formulas $\tilde{\alpha}_{1}(z_{1}),\ldots ,%
\tilde{\alpha}_{k}(z_{1})$. Now it is easy to check that the formula%
\begin{equation*}
\bigwedge_{j=1}^{r}\left( \pi _{j}\rightarrow \alpha _{j}\right) \wedge
\bigwedge_{j=1}^{l}\bigwedge_{t=1}^{k}\left( \bigwedge_{i=1}^{u_{j}}\beta
_{i}^{j}\rightarrow \tilde{\alpha}_{t}\right)
\end{equation*}%
defines $R$ in $\mathcal{K}$.
\end{proof}

Here is the functional version of Theorem \ref{open horn para relaciones}.

\begin{theorem}
\label{openHorn}Let $\mathcal{L}\subseteq \mathcal{L}^{\prime }$ be first
order languages and let $f_{1},\ldots ,f_{m}\in \mathcal{L}^{\prime }-%
\mathcal{L}$ be $n$-ary function symbols. For a class $\mathcal{K}$ of $%
\mathcal{L}^{\prime }$-structures, the following are equivalent:

\begin{enumerate}
\item[(1)] There is a formula in $\mathrm{OpHorn}(\mathcal{L})$ (resp. $%
\mathrm{OpStHorn}(\mathcal{L})$) which defines $\vec{f}$ in $\mathcal{K}$.

\item[(2)] For all $\mathbf{A},\mathbf{B}\in \mathbb{P}_{u}\mathbb{P}_{%
\mathrm{fin}}(\mathcal{K})$, all $\mathbf{A}_{0}\leq \mathbf{A}_{\mathcal{L}%
} $, $\mathbf{B}_{0}\leq \mathbf{B}_{\mathcal{L}}$, all isomorphisms $\sigma
:\mathbf{A}_{0}\rightarrow \mathbf{B}_{0}$, and all $a_{1},\ldots ,a_{n}\in
A_{0}$ such that $f_{1}^{\mathbf{A}}(\vec{a}),\ldots ,f_{m}^{\mathbf{A}}(%
\vec{a})\in A_{0}$, we have that $\sigma (f_{i}^{\mathbf{A}}(\vec{a}%
))=f_{i}^{\mathbf{B}}(\sigma (a_{1}),\ldots ,\sigma (a_{n}))$ for all $i\in
\{1,\ldots ,m\}$. (For all $\mathbf{A}\in \mathcal{K}_{\mathcal{L}}$, and
every trivial substructure $\{a\}$, we have $f_{i}^{\mathbf{A}}(a,\ldots
,a)=a$ for all $i\in \{1,\ldots ,m\}$.)
\end{enumerate}

\noindent Moreover, if $\mathcal{K}_{\mathcal{L}}$ has finitely many
isomorphism types of $(n+m)$-generated substructures and each one is finite,
then we can remove the operator $\mathbb{P}_{u}$ in (2).
\end{theorem}

\begin{proof}
This can be proved applying Theorem \ref{open horn para relaciones} in the
same way as we applied \ref{(positive) open para relaciones} to prove
Theorem \ref{(positive) open}.
\end{proof}

\begin{corollary}
Let $\mathcal{K}$ be any class of $\mathcal{L}$-algebras contained in a
locally finite variety. Suppose $\mathbf{A}\rightarrow f^{\mathbf{A}}$ is a
map which assigns to each $\mathbf{A}\in \mathcal{K}$ an $n$-ary operation $%
f^{\mathbf{A}}:A^{n}$ $\rightarrow A$. The following are equivalent:

\begin{enumerate}
\item[(1)] There is a conjunction of $\mathcal{L}$-formulas of the form $%
\left( \bigwedge_{i}p_{i}=q_{i}\right) \rightarrow r=s$ which defines $f$ in 
$\mathcal{K}$.

\item[(2)] The following conditions hold:

\begin{enumerate}
\item[(a)] For all $\mathbf{A}\in \mathcal{K}$ and all $\{a\}\leq \mathbf{A}$%
, we have $f^{\mathbf{A}}(a,\dots ,a)=a$.

\item[(b)] For all $\mathbf{S}_{1}\leq \mathbf{A}_{1}\times \ldots \times 
\mathbf{A}_{k}$, $\mathbf{S}_{2}\leq \mathbf{B}_{1}\times \ldots \times 
\mathbf{B}_{l}$, with $\mathbf{A}_{1},\ldots ,\mathbf{A}_{k},\mathbf{B}%
_{1},\ldots ,$ $\mathbf{B}_{l}\in \mathcal{K}$, all isomorphisms $\sigma :%
\mathbf{S}_{1}\rightarrow \mathbf{S}_{2}$, and all $p_{1},\ldots ,p_{n}\in
S_{1}$ such that $f^{\mathbf{A}_{1}}\times \ldots \times f^{\mathbf{A}%
_{k}}(p_{1},\ldots ,p_{n})\in S_{1}$, we have\newline
$\sigma (f^{\mathbf{A}_{1}}\times \ldots \times f^{\mathbf{A}%
_{k}}(p_{1},\ldots ,p_{n}))=f^{\mathbf{B}_{1}}\times \ldots \times f^{%
\mathbf{B}_{l}}(\sigma (p_{1}),\dots ,\sigma (p_{n}))$.
\end{enumerate}
\end{enumerate}
\end{corollary}

\begin{proof}
Apply the moreover part of Theorem \ref{openHorn} to the class $\{(\mathbf{A}%
,f^{\mathbf{A}}):\mathbf{A}\in \mathcal{K}\}$.
\end{proof}

\section{\label{Sec: definibilidad por atomicas}Definability by conjunctions
of atomic formulas}

\begin{theorem}
\label{conjuncion de atomicas para relaciones}Let $\mathcal{L}\subseteq 
\mathcal{L}^{\prime }$ be first order languages and let $R\in \mathcal{L}%
^{\prime }-\mathcal{L}$ be an $n$-ary relation symbol. Let $\mathcal{K}$ be
any class of $\mathcal{L}^{\prime }$-structures The following are equivalent:

\begin{enumerate}
\item[(1)] There is a formula in $\left[ \bigwedge \mathrm{At}(\mathcal{L})%
\right] $ which defines $R$ in $\mathcal{K}$.

\item[(2)] For all $\mathbf{A}\in \mathbb{P}_{u}\mathbb{P}_{\mathrm{fin}}(%
\mathcal{K})$, all $\mathbf{B}\in \mathbb{P}_{u}(\mathcal{K})$, $\mathbf{A}%
_{0}\leq \mathbf{A}_{\mathcal{L}}$, $\mathbf{B}_{0}\leq \mathbf{B}_{\mathcal{%
L}}$, all homomorphisms $\sigma :\mathbf{A}_{0}\rightarrow \mathbf{B}_{0}$,
and all $a_{1},\ldots ,a_{n}\in A_{0}$, we have that $(a_{1},\ldots
,a_{n})\in R^{\mathbf{A}}$ implies $(\sigma (a_{1}),\ldots ,\sigma
(a_{n}))\in R^{\mathbf{B}}$.
\end{enumerate}

\noindent Moreover, if $\mathcal{K}_{\mathcal{L}}$ has finitely many
isomorphism types of $n$-generated substructures and each one is finite,
then we can remove the operator $\mathbb{P}_{u}$ from (2).
\end{theorem}

\begin{proof}
(1)$\Rightarrow $(2). This is analogous to the proof of (1)$\Rightarrow $(2)
in Theorem \ref{open horn para relaciones}.

(2)$\Rightarrow $(1). Note that (2) holds also when $\mathbf{B}$ is in $%
\mathbb{ISPP}_{u}(\mathcal{K})$. Since $\mathbb{P}_{u}\mathbb{P}_{\mathrm{fin%
}}(\mathcal{K})\subseteq \mathbb{ISPP}_{u}(\mathcal{K})$, applying Theorem %
\ref{(positive) open para relaciones} to the class $\mathbb{P}_{\mathrm{fin}%
}(\mathcal{K})$ we have that there is an $\mathcal{L}$-formula $\varphi =\pi
_{1}\vee \ldots \vee \pi _{k}$, with each $\pi _{i}$ a conjunction of atomic
formulas, which defines $R$ in $\mathbb{P}_{\mathrm{fin}}(\mathcal{K})$.
Using the same argument as in the proof of Theorem \ref{open horn para
relaciones} we can prove that there is $j$ such that $\pi _{j}$ defines $R$
in $\mathcal{K}$.

The proof of the moreover part is similar to the corresponding part of the
proof of Theorem \ref{open horn para relaciones}.
\end{proof}

\begin{corollary}
\label{conjuncion de atomicas para relaciones caso localmente finito}Let $%
\mathcal{K}$ be a class of $\mathcal{L}$-algebras contained in a locally
finite variety. Suppose $\mathbf{A}\rightarrow R^{\mathbf{A}}$ is a map
which assigns to each $\mathbf{A}\in \mathcal{K}$ an $n$-ary relation $R^{%
\mathbf{A}}\subseteq A^{n}$. The following are equivalent:

\begin{enumerate}
\item[(1)] There is a formula in $\left[ \bigwedge \mathrm{At}(\mathcal{L})%
\right] $ which defines $R$ in $\mathcal{K}$.

\item[(2)] For all $k\in \mathbb{N}$, all $\mathbf{A}_{1},\ldots ,\mathbf{A}%
_{k},\mathbf{A}\in \mathcal{K}$, all $\mathbf{S}\leq \mathbf{A}_{1}\times
\ldots \times \mathbf{A}_{k}$, all homomorphisms $\sigma :\mathbf{S}%
\rightarrow \mathbf{A}$, and all $p_{1},\dots ,p_{n}\in S$, we have that $%
(p_{1},\ldots ,p_{n})\in R^{\mathbf{A}_{1}}\times \ldots \times R^{\mathbf{A}%
_{k}}$ implies $(\sigma (p_{1}),\dots ,\sigma (p_{n}))\in R^{\mathbf{A}}$.
\end{enumerate}
\end{corollary}

\begin{proof}
Apply the moreover part of Theorem \ref{conjuncion de atomicas para
relaciones} to the class $\{(\mathbf{A},R^{\mathbf{A}}):\mathbf{A}\in 
\mathcal{K}\}$.
\end{proof}

Next are the results for definability of functions with conjunction of
atomic formulas.

\begin{theorem}
\label{conjuncion de atomicas}Let $\mathcal{L}\subseteq \mathcal{L}^{\prime
} $ be first order languages and let $f_{1},\ldots ,f_{m}\in \mathcal{L}%
^{\prime }-\mathcal{L}$ be $n$-ary function symbols. For a class $\mathcal{K}
$ of $\mathcal{L}^{\prime }$-structures, the following are equivalent

\begin{enumerate}
\item[(1)] There is a formula $\varphi $ in $\left[ \bigwedge \mathrm{At}(%
\mathcal{L})\right] $ which defines $\vec{f}$ in $\mathcal{K}$.

\item[(2)] For all $\mathbf{A}\in \mathbb{P}_{u}\mathbb{P}_{\mathrm{fin}}(%
\mathcal{K})$, all $\mathbf{B}\in \mathbb{P}_{u}(\mathcal{K})$, $\mathbf{A}%
_{0}\leq \mathbf{A}_{\mathcal{L}}$, $\mathbf{B}_{0}\leq \mathbf{B}_{\mathcal{%
L}}$, all homomorphisms $\sigma :\mathbf{A}_{0}\rightarrow \mathbf{B}_{0}$,
and all $a_{1},\ldots ,a_{n}\in A_{0}$ such that $f_{1}^{\mathbf{A}}(\vec{a}%
),\ldots ,f_{m}^{\mathbf{A}}(\vec{a})\in A_{0}$, we have $\sigma (f_{i}^{%
\mathbf{A}}(\vec{a}))=f_{i}^{\mathbf{B}}(\sigma (a_{1}),\ldots ,\sigma
(a_{n}))$ for all $i\in \{1,\ldots ,m\}$.
\end{enumerate}

\noindent Moreover, if $\mathcal{K}_{\mathcal{L}}$ has finitely many
isomorphism types of $(n+m)$-generated substructures and each one is finite,
then we can remove the operator $\mathbb{P}_{u}$ from (2).
\end{theorem}

\begin{proof}
This can be proved applying Theorem \ref{conjuncion de atomicas para
relaciones} in the same way as we applied \ref{(positive) open para
relaciones} to prove Theorem \ref{(positive) open}.
\end{proof}

\begin{corollary}
Let $\mathcal{K}$ be a class of $\mathcal{L}$-algebras contained in a
locally finite variety. Suppose $\mathbf{A}\rightarrow f^{\mathbf{A}}$ is a
map which assigns to each $\mathbf{A}\in \mathcal{K}$ an $n$-ary operation $%
f^{\mathbf{A}}:A^{n}$ $\rightarrow A$. The following are equivalent:

\begin{enumerate}
\item[(1)] There is a formula in $\left[ \bigwedge \mathrm{At}(\mathcal{L})%
\right] $ which defines $f$ in $\mathcal{K}$.

\item[(2)] For all $k\in \mathbb{N}$, all $\mathbf{A}_{1},\ldots ,\mathbf{A}%
_{k},\mathbf{A}\in \mathcal{K}$, all $\mathbf{S}\leq \mathbf{A}_{1}\times
\ldots \times \mathbf{A}_{k}$, all homomorphisms $\sigma :\mathbf{S}%
\rightarrow \mathbf{A}$, and all $p_{1},\ldots ,p_{n}\in S$ such that $f^{%
\mathbf{A}_{1}}\times \ldots \times f^{\mathbf{A}_{k}}(p_{1},\ldots
,p_{n})\in S$, we have $\sigma (f^{\mathbf{A}_{1}}\times \ldots \times f^{%
\mathbf{A}_{k}}(p_{1},\ldots ,p_{n}))=f^{\mathbf{A}}(\sigma (p_{1}),\dots
,\sigma (p_{n}))$.
\end{enumerate}
\end{corollary}

\begin{proof}
Apply the moreover part of Theorem \ref{conjuncion de atomicas} to the class 
$\{(\mathbf{A},f^{\mathbf{A}}):\mathbf{A}\in \mathcal{K}\}$.
\end{proof}

\subsection*{Applications to definable principal congruences}

A variety $\mathcal{V}$ in a language $\mathcal{L}$ has \emph{equationally
definable principal congruences }if there exists a formula $\varphi
(x,y,z,w)\in \left[ \bigwedge \mathrm{At}(\mathcal{L})\right] $ such that%
\begin{equation*}
\theta ^{\mathbf{A}}(a,b)=\{(c,d):\mathcal{V}\vDash \varphi (a,b,c,d)\}
\end{equation*}%
for any $a,b\in A$, $\mathbf{A}\in \mathcal{V}$. (This notion is called 
\emph{equationally definable principal congruences in the restricted sense}
in \cite{bu}.) The variety $\mathcal{V}$ has the \emph{Fraser-Horn property}
if for every $\mathbf{A}_{1},\mathbf{A}_{2}\in \mathcal{V}$ and $\theta \in 
\mathrm{Con}(\mathbf{A}_{1}\times \mathbf{A}_{2})$, there are $\theta
_{1}\in \mathrm{Con}(\mathbf{A}_{1})$ and $\theta _{2}\in \mathrm{Con}(%
\mathbf{A}_{2})$ such that $\theta =\theta _{1}\times \theta _{2}$ (i.e.,
algebras in $\mathcal{V}$ do not have skew congruences).

Here is an interesting application of Theorem \ref{conjuncion de atomicas
para relaciones}.

\begin{proposition}
\label{aplicacion a dpc 2}Let $\mathcal{V}$ be a variety with definable
principal congruences. The following are equivalent:

\begin{enumerate}
\item[(1)] $\mathcal{V}$ has equationally definable principal congruences.

\item[(2)] $\mathcal{V}$ has the congruence extension property and the
Fraser-Horn property.
\end{enumerate}
\end{proposition}

\begin{proof}
It is well known (see \cite{bu}) that:

\begin{enumerate}
\item[(i)] A variety has the Fraser-Horn property iff for all $n\in \mathbb{N%
}$, all $\mathbf{A}_{1},\dots ,\mathbf{A}_{n}\in \mathcal{V}$, and all $%
(a_{1},\ldots ,a_{n}),$ $(b_{1},\ldots ,b_{n})\in \mathbf{A}_{1}\times
\ldots \times \mathbf{A}_{n}$ we have%
\begin{equation*}
\theta ^{\mathbf{A}_{1}\times \ldots \times \mathbf{A}_{n}}((a_{1},\ldots
,a_{n}),(b_{1},\ldots ,b_{n}))=\theta ^{\mathbf{A}_{1}}(a_{1},b_{1})\times
\ldots \times \theta ^{\mathbf{A}_{n}}(a_{n},b_{n})\text{.}
\end{equation*}
\end{enumerate}

(1)$\Rightarrow $(2). By Proposition \ref{aplicacion a dpc} we have that $%
\mathcal{V}$ has the congruence extension property. Also we note that (i)
and (1) imply that $\mathcal{V}$ has the Fraser-Horn property.

(2)$\Rightarrow $(1). Let $\mathcal{L}$ be the language of $\mathcal{V}$. We
will use the following well known fact.

\begin{enumerate}
\item[(ii)] If $\sigma :\mathbf{A}\rightarrow \mathbf{B}$ is a homomorphism,
then $(a,b)\in \theta ^{\mathbf{A}}\left( c,d\right) $ implies $(\sigma
(a),\sigma (b))\in \theta ^{\mathbf{B}}\left( \sigma \left( c\right) ,\sigma
\left( d\right) \right) $.
\end{enumerate}

\noindent For each $\mathbf{A}\in \mathcal{V}$, let%
\begin{equation*}
R^{\mathbf{A}}=\{(a,b,c,d):(a,b)\in \theta ^{\mathbf{A}}\left( c,d\right) \}%
\text{,}
\end{equation*}%
and define%
\begin{equation*}
\mathcal{K}=\{\left( \mathbf{A},R^{\mathbf{A}}\right) :\mathbf{A}\in 
\mathcal{V}\}\text{.}
\end{equation*}%
Since $\mathcal{V}$ has definable principal congruences, $\mathcal{K}$ is a
first order class and hence $\mathbb{P}_{u}(\mathcal{K})\subseteq \mathcal{K}
$. Since $\mathcal{V}$ has the Fraser-Horn property, (i) says that $\mathbb{P%
}_{\mathrm{fin}}(\mathcal{K})\subseteq \mathcal{K}$. Thus, in order to prove
that (2) of Theorem \ref{conjuncion de atomicas para relaciones} holds we
need to check that:

\begin{enumerate}
\item[(iii)] For all $\mathbf{A},\mathbf{B}\in \mathcal{K}$, all $\mathbf{A}%
_{0}\leq \mathbf{A}_{\mathcal{L}}$, $\mathbf{B}_{0}\leq \mathbf{B}_{\mathcal{%
L}}$, all homomorphisms $\sigma :\mathbf{A}_{0}\rightarrow \mathbf{B}_{0}$
and all $a,b,c,d\in A_{0}$, we have that $(a,b,c,d)\in R^{\mathbf{A}}$
implies $(\sigma (a),\sigma (b),\sigma (c),\sigma (d))\in R^{\mathbf{B}}$.
\end{enumerate}

\noindent Or, equivalently

\begin{enumerate}
\item[(iv)] For all $\mathbf{A},\mathbf{B}\in \mathcal{Q}$, all $\mathbf{A}%
_{0}\leq \mathbf{A}_{\mathcal{L}}$, $\mathbf{B}_{0}\leq \mathbf{B}_{\mathcal{%
L}}$, all homomorphisms $\sigma :\mathbf{A}_{0}\rightarrow \mathbf{B}_{0}$
and all $a,b,c,d\in A_{0}$, we have that $(a,b)\in \theta _{\mathcal{Q}}^{%
\mathbf{A}}\left( c,d\right) $ implies $(\sigma (a),\sigma (b))\in \theta _{%
\mathcal{Q}}^{\mathbf{B}}(\sigma (c),\sigma (d))$.
\end{enumerate}

\noindent Since $\mathcal{V}$ has the congruence extension property we can
replace in (iv) the occurrence of "$(a,b)\in \theta ^{\mathbf{A}}\left(
c,d\right) $" by "$(a,b)\in \theta ^{\mathbf{A}_{0}}(c,d)$". Hence (iv)
follows from (ii).
\end{proof}

\begin{corollary}
\label{aplicacion a dpc 3}A locally finite variety with the congruence
extension property and the Fraser-Horn property has equationally definable
principal congruences.
\end{corollary}

\begin{proof}
Combine Corollary \ref{aplicacion a dpc 1} with Proposition \ref{aplicacion
a dpc 2}.
\end{proof}

It is worth mentioning that in the terminology of \cite{bu}, a variety is
said to have equationally definable principal congruences if there is a
formula of the form $\exists \bigwedge p=q$ which defines the principal
congruences. Thus, Theorem 4 of \cite{bu} is not in contradiction with
Proposition \ref{aplicacion a dpc 2}, in fact it coincides with Corollary %
\ref{aplicacion a dpc 4} of the next section.

\subsection*{Two translation results}

As another application of Theorem \ref{conjuncion de atomicas para
relaciones}, we obtain a model theoretic proof of the following translation
result.

\begin{proposition}[{\protect\cite[Thm 2.3]{cz-dz}}]
\label{traduccion}Let $\mathcal{K}$ be a universal class of $\mathcal{L}$%
-algebras such that $\mathcal{K}\subseteq \mathcal{Q}_{RFSI}$, for some
relatively congruence distributive quasivariety $\mathcal{Q}$. Let $\varphi (%
\vec{x})\in \left[ \bigvee \bigwedge \mathrm{At}(\mathcal{L})\right] $. Then
there are $\mathcal{L}$-terms $p_{i},q_{i}$, $i=1,\ldots ,r$ such that%
\begin{equation*}
\mathcal{K}\vDash \varphi (\vec{x})\leftrightarrow \left(
\dbigwedge_{i=1}^{r}p_{i}(\vec{x})=q_{i}(\vec{x})\right) \text{.}
\end{equation*}
\end{proposition}

\begin{proof}
We start by proving the proposition for the formula%
\begin{equation*}
\varphi =\left( x_{1}=x_{2}\vee x_{3}=x_{4}\right) \text{.}
\end{equation*}%
Given $\mathbf{A}\in \mathcal{K}$, let%
\begin{equation*}
R^{\mathbf{A}}=\{\vec{a}\in A^{4}:\mathbf{A}\vDash \varphi (\vec{a}%
)\}=\{(a,b,c,d)\in A^{4}:a=b\text{ or }c=d\}\text{.}
\end{equation*}%
Define $\mathcal{K}^{\prime }=\{(\mathbf{A},R^{\mathbf{A}}):\mathbf{A}\in 
\mathcal{K}\}$. Note that $\mathcal{K}^{\prime }$ is universal. We aim to
apply (2) of Theorem \ref{conjuncion de atomicas para relaciones}, so we
need to show that:

\begin{itemize}
\item For all $\mathbf{A}\in \mathbb{P}_{u}\mathbb{P}_{\mathrm{fin}}(%
\mathcal{K}^{\prime })$, all $\mathbf{B}\in \mathbb{P}_{u}(\mathcal{K}%
^{\prime })$, all $\mathbf{A}_{0}\leq \mathbf{A}_{\mathcal{L}}$, $\mathbf{B}%
_{0}\leq \mathbf{B}_{\mathcal{L}}$, all homomorphisms $\sigma :\mathbf{A}%
_{0}\rightarrow \mathbf{B}_{0}$, and all $a,b,c,d\in A_{0}$, we have that $%
(a,b,c,d)\in R^{\mathbf{A}}$ implies $(\sigma (a),\sigma (b),\sigma
(c),\sigma (d))\in R^{\mathbf{B}}$.
\end{itemize}

Since $\mathcal{K}^{\prime }$ is universal, we can suppose that $\mathbf{B}%
\in \mathcal{K}^{\prime }$, $\mathbf{B}_{0}=\mathbf{B}_{\mathcal{L}}$ and $%
\sigma $ is onto. Also, as $\mathbb{P}_{u}\mathbb{P}_{\mathrm{fin}}(\mathcal{%
K}^{\prime })\subseteq \mathbb{ISPP}_{u}(\mathcal{K}^{\prime })\subseteq 
\mathbb{ISP}(\mathcal{K}^{\prime })$, we may assume that $\mathbf{A}\leq \Pi
\{\mathbf{A}_{i}:i\in I\}$ is a subdirect product with each $\mathbf{A}_{i}$
in $\mathcal{K}^{\prime }$, and that $\mathbf{A}_{0}=\mathbf{A}_{\mathcal{L}%
} $. Since%
\begin{equation*}
\mathbf{A}/\ker \sigma \simeq \mathbf{B}\in \mathcal{K}^{\prime }
\end{equation*}%
and $\mathcal{K\subseteq Q}_{RFSI}$, we have that $\ker \sigma $ is a meet
irreducible element of $\mathrm{Con}_{\mathcal{Q}}(\mathbf{A}_{\mathcal{L}})$%
. So, as $\mathrm{Con}_{\mathcal{Q}}(\mathbf{A}_{\mathcal{L}})$ is
distributive, we have that $\ker \sigma $ is a meet prime element of $%
\mathrm{Con}_{\mathcal{Q}}(\mathbf{A}_{\mathcal{L}})$. Let $(a,b,c,d)\in R^{%
\mathbf{A}}$. We prove that $(\sigma (a),\sigma (b),\sigma (c),\sigma
(d))\in R^{\mathbf{B}}$ (i.e., $\sigma (a)=\sigma (b)$ or $\sigma (c)=\sigma
(d)$). For $i\in I$ let $\pi _{i}:\mathbf{A}\rightarrow \mathbf{A}_{i}$ be
the canonical projection. Note that $\ker \pi _{i}\in \mathrm{Con}_{\mathcal{%
Q}}(\mathbf{A}_{\mathcal{L}})$ since $\mathbf{A}/\ker \pi _{i}\simeq \mathbf{%
A}_{i}\in \mathcal{K}^{\prime }$. From $(a,b,c,d)\in R^{\mathbf{A}}$, it
follows that either $a(i)=b(i)$ or $c(i)=d(i)$ for every $i\in I$. Thus%
\begin{equation*}
\bigcap_{(a,b)\in \ker \pi _{i}}\ker \pi _{i}\cap \bigcap_{(c,d)\in \ker \pi
_{i}}\ker \pi _{i}=\Delta ^{\mathbf{A}_{\mathcal{L}}}\subseteq \ker \sigma 
\text{.}
\end{equation*}%
Since $\ker \sigma $ is meet prime, this implies that either%
\begin{equation*}
\bigcap_{(a,b)\in \ker \pi _{i}}\ker \pi _{i}\subseteq \ker \sigma
\end{equation*}%
or%
\begin{equation*}
\bigcap_{(c,d)\in \ker \pi _{i}}\ker \pi _{i}\subseteq \ker \sigma \text{,}
\end{equation*}%
which implies $\sigma (a)=\sigma (b)$ or $\sigma (c)=\sigma (d)$. This
concludes the proof for the case $\varphi =\left( x_{1}=x_{2}\vee
x_{3}=x_{4}\right) $. The case in which $\varphi $ is the formula $%
x_{1}=y_{1}\vee x_{2}=y_{2}\vee \ldots \vee x_{n}=y_{n}$ can be proved in a
similar manner. Now, the general case follows easily.
\end{proof}

Strenghtening RFSI to RS allows for the translation of any open formula to a
conjunction of equations over $\mathcal{K}$. This is proved in \cite{ca-va0}
using topological arguments. Combining Proposition \ref{traduccion} and
Theorem \ref{(positive) open para relaciones} we get a very simple proof.

\begin{corollary}
\label{traduccion para simples}Let $\mathcal{K}$ be a universal class of $%
\mathcal{L}$-algebras such that $\mathcal{K}\subseteq \mathcal{Q}_{RS}$, for
some relatively congruence distributive quasivariety $\mathcal{Q}$. Let $%
\varphi (\vec{x})\in \mathrm{Op}(\mathcal{L})$. Then there are $\mathcal{L}$%
-terms $p_{i},q_{i}$, $i=1,\ldots ,r$ such that%
\begin{equation*}
\mathcal{K}\vDash \varphi (\vec{x})\leftrightarrow \left(
\dbigwedge_{i=1}^{r}p_{i}(\vec{x})=q_{i}(\vec{x})\right) \text{.}
\end{equation*}
\end{corollary}

\begin{proof}
We show first that there is $\delta (x,y)\in \left[ \bigvee \bigwedge 
\mathrm{At}(\mathcal{L})\right] $ such that $\mathcal{K}\vDash x\neq
y\leftrightarrow \delta (x,y)$. Given $\mathbf{A}\in \mathcal{K}$, let%
\begin{equation*}
D^{\mathbf{A}}=\{(a,b)\in A^{2}:a\neq b\}\text{,}
\end{equation*}%
and define $\mathcal{K}^{\prime }=\{(\mathbf{A},D^{\mathbf{A}}):\mathbf{A}%
\in \mathcal{K}\}$. Observe that $\mathcal{K}^{\prime }$ is universal. We
want to apply Theorem \ref{(positive) open para relaciones}, thus we need to
check that:

\begin{itemize}
\item For all $\mathbf{A},\mathbf{B}\in \mathbb{P}_{u}(\mathcal{K}^{\prime
}) $, all $\mathbf{A}_{0}\leq \mathbf{A}_{\mathcal{L}}$, $\mathbf{B}_{0}\leq 
\mathbf{B}_{\mathcal{L}}$, all homomorphisms $\sigma :\mathbf{A}%
_{0}\rightarrow \mathbf{B}_{0}$, and all $a,b\in A_{0}$, we have that $a\neq
b$ implies $\sigma (a)\neq \sigma (b)$.
\end{itemize}

\noindent But this is easy. Just note that both $\mathbf{A}_{0}$ and $\func{%
Im}\sigma $ are in $\mathcal{Q}_{RS}$ since $\mathcal{K}$ is universal. So $%
\ker \sigma \in \mathrm{Con}_{\mathcal{Q}}(\mathbf{A}_{0})=\{\Delta ^{%
\mathbf{A}_{0}},\nabla ^{\mathbf{A}_{0}}\}$, and $\ker \sigma \neq \nabla ^{%
\mathbf{A}_{0}}$ because $\func{Im}\sigma $ is simple and thus non-trivial.
It follows that $\sigma $ is one-one. By Theorem \ref{(positive) open para
relaciones} we have a formula in $\left[ \bigvee \bigwedge \mathrm{At}(%
\mathcal{L})\right] $ that defines $D$ in $\mathcal{K}$.

Now Proposition \ref{traduccion} produces $\varepsilon \left( x,y\right) \in %
\left[ \bigwedge \mathrm{At}(\mathcal{L})\right] $ such that%
\begin{equation*}
\mathcal{K}\vDash x\neq y\leftrightarrow \varepsilon (x,y)\text{.}
\end{equation*}%
This, in combination with the fact that disjunctions of equations are
equivalent to conjunctions in $\mathcal{K}$ (again by Proposition \ref%
{traduccion}), lets us translate any open formula to a conjunction of
equations over $\mathcal{K}$.
\end{proof}

The translation results above produce the following interesting corollaries
for a finite algebra.

\begin{corollary}
\label{algebraicas para A con traduccion}Suppose $\mathbf{A}$ is a finite $%
\mathcal{L}$-algebra such that $\mathbb{S}(\mathbf{A})\subseteq \mathcal{Q}%
_{RFSI}$ (resp. $\mathbb{S}(\mathbf{A})\subseteq \mathcal{Q}_{RS}$), for
some relatively congruence distributive quasivariety $\mathcal{Q}$. Let $%
f:A^{n}\rightarrow A$. The following are equivalent:

\begin{enumerate}
\item[(1)] There is a formula in $\left[ \bigwedge \mathrm{At}(\mathcal{L})%
\right] $ which defines $f$ in $\mathbf{A}$.

\item[(2)] For all $\mathbf{S}_{1},\mathbf{S}_{2}\leq \mathbf{A}$, all
homomorphisms (resp. isomorphisms) $\sigma :\mathbf{S}_{1}\rightarrow 
\mathbf{S}_{2}$, and all $a_{1},\ldots ,a_{n}\in S_{1}$ such that $f^{%
\mathbf{A}}(\vec{a})\in S_{1}$, we have $\sigma (f^{\mathbf{A}}(\vec{a}))=f^{%
\mathbf{A}}(\sigma (a_{1}),\ldots ,\sigma (a_{n}))$.
\end{enumerate}
\end{corollary}

\begin{proof}
(1)$\Rightarrow $(2). Apply (1)$\Rightarrow $(2) of Theorem \ref{conjuncion
de atomicas}.

(2)$\Rightarrow $(1). By Corollary \ref{open para localy finite} there is a
formula $\varphi (\vec{x},z)\in \left[ \bigvee \bigwedge \mathrm{At}(%
\mathcal{L})\right] $ (resp. $\varphi (\vec{x},z)\in \mathrm{Op}(\mathcal{L}%
) $) which defines $f$ in $\mathbb{IS}(\mathbf{A})$. Now use Proposition \ref%
{traduccion} (resp. Corollary \ref{traduccion para simples}) to obtain a
conjunction of equations equivalent to $\varphi $ over $\mathbb{IS}(\mathbf{A%
})$.
\end{proof}

\begin{corollary}
Suppose $\mathbf{A}$ is a finite $\mathcal{L}$-algebra such that $\mathbb{S}(%
\mathbf{A})\subseteq \mathcal{Q}_{RFSI}$ (resp. $\mathbb{S}(\mathbf{A}%
)\subseteq \mathcal{Q}_{RS}$), for some relatively congruence distributive
quasivariety $\mathcal{Q}$. Let $R\subseteq A^{n}$. The following are
equivalent:

\begin{enumerate}
\item[(1)] There is a formula in $\left[ \bigwedge \mathrm{At}(\mathcal{L})%
\right] $ which defines $R$ in $\mathbf{A}$.

\item[(2)] For all $\mathbf{S}_{1},\mathbf{S}_{2}\leq \mathbf{A}$, all
homomorphisms (resp. isomorphisms) $\sigma :\mathbf{S}_{1}\rightarrow 
\mathbf{S}_{2}$, and all $a_{1},\ldots ,a_{n}\in S_{1}$, we have that $%
(a_{1},\ldots ,a_{n})\in R$ implies\newline
$(\sigma (a_{1}),\ldots ,\sigma (a_{n}))\in R$.
\end{enumerate}
\end{corollary}

\section{Definability by existential formulas}

\begin{lemma}
\label{embedding a ultraproducto}Let $\mathbf{A},\mathbf{B}$ be $\mathcal{L}$%
-structures. Suppose for every sentence $\varphi \in \left[ \exists
\bigwedge \mathrm{At}(\mathcal{L})\right] $ (resp. $\varphi \in \left[
\exists \bigwedge \mathrm{\pm At}(\mathcal{L})\right] $) we have that $%
\mathbf{A}\vDash \varphi $ implies $\mathbf{B}\vDash \varphi $. Then there
is a homomorphism (resp. embedding) from $\mathbf{A}$ into an ultrapower of $%
\mathbf{B}$.
\end{lemma}

\begin{proof}
Let $\mathcal{L}_{A}=\mathcal{L}\cup A$, where each element of $A$ is added
as a new constant symbol. Define%
\begin{equation*}
\Delta =\{\alpha \left( \vec{a}\right) :\alpha \in \lbrack \bigwedge \mathrm{%
At}(\mathcal{L})]\text{ and }\mathbf{A}\vDash \alpha (\vec{a})\}\text{,}
\end{equation*}%
i.e., $\Delta $ is the positive atomic diagram of $\mathbf{A}$. Let $I=%
\mathcal{P}_{\mathrm{fin}}\left( \Delta \right) $, and observe that for
every $i\in I$ there is an expansion $\mathbf{B}_{i}$ of $\mathbf{B}$ to $%
\mathcal{L}_{A}$ such that $\mathbf{B}_{i}\vDash i$. Now take an ultrafilter 
$u$ over $I$ such that for each $i\in I$ the set $\{j\in I:i\subseteq j\}$
is in $u$. Let $\mathbf{U}=\left. \prod \mathbf{B}_{i}\right/ u$, and notice
that $\mathbf{U}\vDash \Delta $. Thus $a\mapsto a^{\mathbf{U}}$ is an
homomorphism from $\mathbf{A}$ into $\mathbf{U}_{\mathcal{L}}=\mathbf{B}%
^{I}/u$.

The same proof works for the embedding version of the lemma by taking $%
\Delta =\{\alpha \left( \vec{a}\right) :\alpha \in \lbrack \mathrm{\pm }%
\bigwedge \mathrm{At}(\mathcal{L})]$ and $\mathbf{A}\vDash \alpha (\vec{a}%
)\} $.
\end{proof}

\begin{theorem}
\label{caso existencial para relaciones}Let $\mathcal{L}\subseteq \mathcal{L}%
^{\prime }$ be first order languages and let $R\in \mathcal{L}^{\prime }-%
\mathcal{L}$ be an $n$-ary relation symbol. Let $\mathcal{K}$ be any class
of $\mathcal{L}^{\prime }$-structures. Then

\begin{enumerate}
\item[(1)] The following are equivalent:

\begin{enumerate}
\item[(a)] There is a formula in $\left[ \exists \mathrm{Op}(\mathcal{L})%
\right] $ which defines $R$ in $\mathcal{K}$.

\item[(b)] For all $\mathbf{A},\mathbf{B}\in \mathbb{P}_{u}(\mathcal{K})$
and all embeddings $\sigma :\mathbf{A}_{\mathcal{L}}\rightarrow \mathbf{B}_{%
\mathcal{L}}$, we have that $\sigma :(\mathbf{A}_{\mathcal{L}},R^{\mathbf{A}%
})\rightarrow (\mathbf{B}_{\mathcal{L}},R^{\mathbf{B}})$ is a homomorphism.
\end{enumerate}

\item[(2)] The following are equivalent:

\begin{enumerate}
\item[(a)] There is a formula in $\left[ \exists \mathrm{OpHorn}(\mathcal{L})%
\right] $ which defines $R$ in $\mathcal{K}$.

\item[(b)] For all $\mathbf{A},\mathbf{B}\in \mathbb{P}_{u}\mathbb{P}_{%
\mathrm{fin}}(\mathcal{K})$ and all embeddings $\sigma :\mathbf{A}_{\mathcal{%
L}}\rightarrow \mathbf{B}_{\mathcal{L}}$, we have that $\sigma :(\mathbf{A}_{%
\mathcal{L}},R^{\mathbf{A}})\rightarrow (\mathbf{B}_{\mathcal{L}},R^{\mathbf{%
B}})$ is a homomorphism.
\end{enumerate}

\item[(3)] The following are equivalent:

\begin{enumerate}
\item[(a)] There is a formula in $\left[ \exists \bigvee \bigwedge \mathrm{At%
}(\mathcal{L})\right] $ which defines $R$ in $\mathcal{K}$.

\item[(b)] For all $\mathbf{A},\mathbf{B}\in \mathbb{P}_{u}(\mathcal{K})$
and all homomorphisms $\sigma :\mathbf{A}_{\mathcal{L}}\rightarrow \mathbf{B}%
_{\mathcal{L}}$, we have that $\sigma :(\mathbf{A}_{\mathcal{L}},R^{\mathbf{A%
}})\rightarrow (\mathbf{B}_{\mathcal{L}},R^{\mathbf{B}})$ is a homomorphism.
\end{enumerate}

\item[(4)] The following are equivalent:

\begin{enumerate}
\item[(a)] There is a formula in $\left[ \exists \bigwedge \mathrm{At}(%
\mathcal{L})\right] $ which defines $R$ in $\mathcal{K}$.

\item[(b)] For all $\mathbf{A},\mathbf{B}\in \mathbb{P}_{u}\mathbb{P}_{%
\mathrm{fin}}(\mathcal{K})$ and all homomorphisms $\sigma :\mathbf{A}_{%
\mathcal{L}}\rightarrow \mathbf{B}_{\mathcal{L}}$, we have that $\sigma :(%
\mathbf{A}_{\mathcal{L}},R^{\mathbf{A}})\rightarrow (\mathbf{B}_{\mathcal{L}%
},R^{\mathbf{B}})$ is a homomorphism.
\end{enumerate}
\end{enumerate}

\noindent Moreover, if modulo isomorphism, $\mathcal{K}$ is a finite class
of finite structures, then we can remove the operator $\mathbb{P}_{u}$ from
the (b) items.
\end{theorem}

\begin{proof}
(1). (a)$\Rightarrow $(b). Note that if $\varphi \left( \vec{x}\right) \in %
\left[ \exists \mathrm{Op}(\mathcal{L})\right] $ defines $R$ in $\mathcal{K}$%
, then $\varphi $ defines $R$ in $\mathbb{P}_{u}(\mathcal{K})$ as well. Now
use that $\varphi $ is preserved by embeddings.

(b)$\Rightarrow $(a). For $\mathbf{A}\in \mathbb{P}_{u}(\mathcal{K})$ and $%
\vec{a}\in R^{\mathbf{A}}$, let%
\begin{equation*}
\Gamma ^{\vec{a},\mathbf{A}}=\{\alpha (\vec{x}):\alpha \in \left[ \exists
\dbigwedge \mathrm{\pm At}(\mathcal{L})\right] \text{ and }\mathbf{A}\vDash
\alpha (\vec{a})\}\text{.}
\end{equation*}%
Suppose $\mathbf{B}\in \mathbb{P}_{u}(\mathcal{K})$ and $\mathbf{B}\vDash
\Gamma ^{\vec{a},\mathbf{A}}(\vec{b})$. We claim that $\vec{b}\in R^{\mathbf{%
B}}$. Let $\mathcal{\tilde{L}}$ be the result of adding $n$ new constant
symbols to $\mathcal{L}$. Note that every sentence of $\left[ \exists
\dbigwedge \mathrm{\pm At}(\mathcal{\tilde{L}})\right] $ which holds in $(%
\mathbf{A}_{\mathcal{L}},\vec{a})$ holds in $(\mathbf{B}_{\mathcal{L}},\vec{b%
})$, which by Lemma \ref{embedding a ultraproducto} says that there is an
embedding from $(\mathbf{A}_{\mathcal{L}},\vec{a})$ into an ultrapower $(%
\mathbf{B}_{\mathcal{L}},\vec{b})^{I}/u$. Since $\mathbb{P}_{u}\mathbb{P}%
_{u}(\mathcal{K})\subseteq \mathbb{P}_{u}(\mathcal{K})$, (b) says that $%
\left( b_{1}/u,\dots ,b_{n}/u\right) \in R^{\mathbf{B}^{I}/u}$, which yields 
$\vec{b}\in R^{\mathbf{B}}$. So we have that%
\begin{equation*}
\mathbb{P}_{u}(\mathcal{K})\vDash \left( \dbigvee\nolimits_{\mathbf{A}\in 
\mathbb{P}_{u}(\mathcal{K})\text{, }\vec{a}\in R^{\mathbf{A}%
}}\dbigwedge\nolimits_{\alpha \in \Gamma ^{\vec{a},\mathbf{A}}}\alpha (\vec{x%
})\right) \leftrightarrow R(\vec{x})\text{,}
\end{equation*}%
and compactness produces the formula.

(2). (a)$\Rightarrow $(b). Observe that if $\varphi \left( \vec{x}\right)
\in \left[ \exists \mathrm{OpHorn}(\mathcal{L})\right] $ defines $R$ in $%
\mathcal{K}$, then $\varphi $ also defines $R$ in $\mathbb{P}_{u}\mathbb{P}_{%
\mathrm{fin}}(\mathcal{K})$. Next use that $\varphi $ is preserved by
embeddings.

(b)$\Rightarrow $(a). By (1) we have that there is $\varphi \in \left[
\exists \mathrm{Op}(\mathcal{L})\right] $ which defines $R$ in $\mathbb{P}_{%
\mathrm{fin}}(\mathcal{K})$. Now we can apply a similar argument to that
used in the proof of Theorem \ref{openHorn}, to extract from $\varphi $ a
formula of $\left[ \exists \mathrm{OpHorn}(\mathcal{L})\right] $ which
defines $R$ in $\mathcal{K}$.

(3). (b)$\Leftrightarrow $(a). This is similar to (b)$\Leftrightarrow $(a)
of (1).

(4). (a)$\Rightarrow $(b). Analogous to (a)$\Rightarrow $(b) of (2).

(b)$\Rightarrow $(a). Note that (b) holds also when $\mathbf{B}$ is in $%
\mathbb{PP}_{u}(\mathcal{K})$. Since $\mathbb{P}_{u}\mathbb{P}_{\mathrm{fin}%
}(\mathcal{K})\subseteq \mathbb{PP}_{u}(\mathcal{K})$, applying (b)$%
\Rightarrow $(a) of (3) to the class $\mathbb{P}_{\mathrm{fin}}(\mathcal{K})$
we have that there is a formula $\varphi \in \left[ \exists \bigvee
\bigwedge \mathrm{At}(\mathcal{L})\right] $ which defines $R$ in $\mathbb{P}%
_{\mathrm{fin}}(\mathcal{K})$. Now we can apply a similar argument to that
used in the proof of Theorem \ref{openHorn}, to extract from $\varphi $ a
formula of $\left[ \exists \bigwedge \mathrm{At}(\mathcal{L})\right] $ which
defines $R$ in $\mathcal{K}$.

The moreover part is left to the reader (see the proof of \ref{open horn
para relaciones}).
\end{proof}

We state without proof the functional version of the above theorem.

\begin{theorem}
\label{caso existencial}Let $\mathcal{L}\subseteq \mathcal{L}^{\prime }$ be
first order languages and let $f_{1},\ldots ,f_{m}\in \mathcal{L}^{\prime }-%
\mathcal{L}$ be $n$-ary function symbols. Let $\mathcal{K}$ be a class of $%
\mathcal{L}^{\prime }$-structures. Then we have:

\begin{enumerate}
\item[(1)] The following are equivalent:

\begin{enumerate}
\item[(a)] There is a formula in $\left[ \exists \mathrm{Op}(\mathcal{L})%
\right] $ which defines $\vec{f}$ in $\mathcal{K}$.

\item[(b)] For all $\mathbf{A},\mathbf{B}\in \mathbb{P}_{u}(\mathcal{K})$
and all embeddings $\sigma :\mathbf{A}_{\mathcal{L}}\rightarrow \mathbf{B}_{%
\mathcal{L}}$, we have that $\sigma :(\mathbf{A}_{\mathcal{L}},f_{1}^{%
\mathbf{A}},\ldots ,f_{m}^{\mathbf{A}})\rightarrow (\mathbf{B}_{\mathcal{L}%
},f_{1}^{\mathbf{B}},\ldots ,f_{m}^{\mathbf{B}})$ is an embedding.
\end{enumerate}

\item[(2)] The following are equivalent:

\begin{enumerate}
\item[(a)] There is a formula in $\left[ \exists \mathrm{OpHorn}(\mathcal{L})%
\right] $ which defines $\vec{f}$ in $\mathcal{K}$.

\item[(b)] For all $\mathbf{A},\mathbf{B}\in \mathbb{P}_{u}\mathbb{P}_{%
\mathrm{fin}}(\mathcal{K})$ and all embeddings $\sigma :\mathbf{A}_{\mathcal{%
L}}\rightarrow \mathbf{B}_{\mathcal{L}}$, we have that $\sigma :(\mathbf{A}_{%
\mathcal{L}},f_{1}^{\mathbf{A}},\ldots ,f_{m}^{\mathbf{A}})\rightarrow (%
\mathbf{B}_{\mathcal{L}},f_{1}^{\mathbf{B}},\ldots ,f_{m}^{\mathbf{B}})$ is
an embedding.
\end{enumerate}

\item[(3)] The following are equivalent:

\begin{enumerate}
\item[(a)] There is a formula in $\left[ \exists \bigvee \bigwedge \mathrm{At%
}(\mathcal{L})\right] $ which defines $\vec{f}$ in $\mathcal{K}$.

\item[(b)] For all $\mathbf{A},\mathbf{B}\in \mathbb{P}_{u}(\mathcal{K})$
and all homomorphisms $\sigma :\mathbf{A}_{\mathcal{L}}\rightarrow \mathbf{B}%
_{\mathcal{L}}$, we have that $\sigma :(\mathbf{A}_{\mathcal{L}},f_{1}^{%
\mathbf{A}},\ldots ,f_{m}^{\mathbf{A}})\rightarrow (\mathbf{B}_{\mathcal{L}%
},f_{1}^{\mathbf{B}},\ldots ,f_{m}^{\mathbf{B}})$ is a homomorphism.
\end{enumerate}

\item[(4)] The following are equivalent:

\begin{enumerate}
\item[(a)] There is a formula in $\left[ \exists \bigwedge \mathrm{At}(%
\mathcal{L})\right] $ which defines $\vec{f}$ in $\mathcal{K}$.

\item[(b)] For all $\mathbf{A},\mathbf{B}\in \mathbb{P}_{u}\mathbb{P}_{%
\mathrm{fin}}(\mathcal{K})$ and all homomorphisms $\sigma :\mathbf{A}_{%
\mathcal{L}}\rightarrow \mathbf{B}_{\mathcal{L}}$, we have that $\sigma :(%
\mathbf{A}_{\mathcal{L}},f_{1}^{\mathbf{A}},\ldots ,f_{m}^{\mathbf{A}%
})\rightarrow (\mathbf{B}_{\mathcal{L}},f_{1}^{\mathbf{B}},\ldots ,f_{m}^{%
\mathbf{B}})$ is a homomorphism.
\end{enumerate}
\end{enumerate}

\noindent Moreover, if modulo isomorphism $\mathcal{K}$ is a finite class of
finite structures, then we can remove the operator $\mathbb{P}_{u}$ from the
(b) items.
\end{theorem}

For the case in which $\mathcal{K}=\{\mathbf{A}\}$, with $\mathbf{A}$
finite, (1) and (4) of Theorem \ref{caso existencial} are proved in \cite{kr}%
.

Using (b)$\Rightarrow $(a) of (4) in Theorem \ref{caso existencial para
relaciones} we can prove the following result of \cite{bu} (see the
paragraph below Corollary \ref{aplicacion a dpc 3}).

\begin{corollary}
\label{aplicacion a dpc 4}Let $\mathcal{V}$ be a variety with definable
principal congruences. Let $\mathcal{L}$ be the language of $\mathcal{V}$.
The following are equivalent:

\begin{enumerate}
\item[(1)] There is a formula in $\left[ \exists \bigwedge \mathrm{At}(%
\mathcal{L})\right] $ which defines the principal congruences in $\mathcal{V}
$.

\item[(2)] $\mathcal{V}$ has the Fraser-Horn property.
\end{enumerate}
\end{corollary}

\subsection*{Primitive positive functions}

Functions definable in a finite algebra $\mathbf{A}$ by a formula of the
form $\exists \dbigwedge p=q$ are called \emph{primitive positive functions}
and they have been extensively studied. For the case in which $\mathcal{K}=\{%
\mathbf{A}\}$ for some finite algebra $\mathbf{A}$, (4) of Theorem \ref{caso
existencial} is a well known result \cite{ge}. The translation results of
Section 4 produce the following.

\begin{proposition}
\label{primitive positives para A con traduccion}Suppose $\mathbf{A}$ is a
finite $\mathcal{L}$-algebra such that $\mathbb{S}(\mathbf{A})\subseteq
Q_{RFSI}$ (resp. $\mathbb{S}(\mathbf{A})\subseteq Q_{RS}$), for some
relatively congruence distributive quasivariety $Q$. Let $f:A^{n}\rightarrow
A$. The following are equivalent:

\begin{enumerate}
\item[(1)] $f$ is primitive positive.

\item[(2)] If $\sigma :\mathbf{A}\rightarrow \mathbf{A}$ is a homomorphism
(resp. isomorphism), then $\sigma :(\mathbf{A},f)\rightarrow (\mathbf{A},f)$
is a homomorphism (resp. isomorphism).
\end{enumerate}
\end{proposition}

\begin{proof}
(2)$\Rightarrow $(1) By (3) (resp. (1)) of Theorem \ref{caso existencial} we
have that there is a formula $\exists \vec{u}\ \psi (\vec{u},\vec{x},z)$,
with $\psi \in \left[ \bigvee \bigwedge \mathrm{At}(\mathcal{L})\right] $
(resp. $\left[ \mathrm{Op}(\mathcal{L})\right] $), which defines $f$ in $%
\mathbf{A}$. Since $\mathbb{IS}(\mathbf{A})$ is a universal class,
Proposition \ref{traduccion} (resp. Corollary \ref{traduccion para simples})
says that there is a $\varphi \in \left[ \bigwedge \mathrm{At}(\mathcal{L})%
\right] $ equivalent with $\psi $ over $\mathbb{IS}(\mathbf{A})$. Clearly $%
\exists \vec{u}\ \varphi (\vec{u},\vec{x},z)$ defines $f$ in $\mathbf{A}$.
\end{proof}

\subsubsection{Primitive positive functions in Stone algebras}

As an application of the results in Section \ref{Sec: definibilidad por
atomicas} and the current section we characterize primitive positive
functions and functions definable by formulas of the form $\dbigwedge p=q$
in Stone algebras. If $\mathbf{L}$ is a Stone algebra, let $\rightarrow ^{%
\mathbf{L}}$ denote its Heyting implication, when it does exist. A \emph{%
three valued Heyting algebra} is a Heyting algebra belonging to the variety
generated by the three element Heyting algebra.

\begin{proposition}
\label{allgebraicas de stone algebras copy(1)}Let $\mathbf{L}=(L,\vee
,\wedge ,^{\ast },0,1)$ be a Stone algebra and let $f:L^{n}\rightarrow L$ be
any function. Let $\mathcal{L}=\{\vee ,\wedge ,^{\ast },0,1\}$. Then:

\begin{enumerate}
\item[(1)] If the Heyting implication exists in $\mathbf{L}$, and $(L,\vee
,\wedge ,\rightarrow ^{\mathbf{L}},0,1)$ is a three valued Heyting algebra,
then the following are equivalent:

\begin{enumerate}
\item[(a)] There is a formula in $\left[ \bigwedge \mathrm{At}(\mathcal{L})%
\right] $ which defines $f$ in $\mathbf{L}$.

\item[(b)] There is a formula in $\left[ \exists \bigwedge \mathrm{At}(%
\mathcal{L})\right] $ which defines $f$ in $\mathbf{L}$.

\item[(c)] There is a term of the language $\{\vee ,\wedge ,\rightarrow
,^{\ast },0,1\}$ which represents $f$ in $\mathbf{L}$.
\end{enumerate}

\item[(2)] If $\mathbf{L}$ does not satisfy the hypothesis of (1), then the
following are equivalent:

\begin{enumerate}
\item[(a)] There is a formula in $\left[ \bigwedge \mathrm{At}(\mathcal{L})%
\right] $ which defines $f$ in $\mathbf{L}$.

\item[(b)] There is a formula in $\left[ \exists \bigwedge \mathrm{At}(%
\mathcal{L})\right] $ which defines $f$ in $\mathbf{L}$.

\item[(c)] There is an $\mathcal{L}$-term which represents $f$ in $\mathbf{L}
$.
\end{enumerate}
\end{enumerate}
\end{proposition}

\begin{proof}
Let $\mathbf{3}$ be the three element Stone algebra $(\{0,1/2,1\},\max ,\min
,^{\ast },0,1)$. Let $\mathbf{2}$ denote the subalgebra of $\mathbf{3}$ with
universe $\{0,1\}$, i.e. $\mathbf{2}$ is the two element Boolean algebra.
First we prove a series of claims.\medskip

\noindent \textbf{Claim 1.} Let $f:\{0,1/2,1\}^{n}\rightarrow \{0,1/2,1\}$.
The following are equivalent:

\begin{enumerate}
\item[(i)] There is a formula in $\left[ \bigwedge \mathrm{At}(\mathcal{L})%
\right] $ which defines $f$ in $\mathbf{3}$.

\item[(ii)] There is a formula in $\left[ \exists \bigwedge \mathrm{At}(%
\mathcal{L})\right] $ which defines $f$ in $\mathbf{3}$.

\item[(iii)] $f$ is a term function of $(\mathbf{3},\rightarrow ^{\mathbf{3}%
})$.
\end{enumerate}

\noindent Proof. Since $^{\ast \ast }:\{0,1/2,1\}\rightarrow \{0,1/2,1\}$ is
the only non trivial homomorphism between subalgebras of $\mathbf{3}$,
Corollary \ref{algebraicas para A con traduccion} and Proposition \ref%
{primitive positives para A con traduccion} say that both (i) and (ii) are
equivalent to

\begin{enumerate}
\item[(iv)] $f(x_{1},\ldots ,x_{n})^{\ast \ast }=f(x_{1}^{\ast \ast },\ldots
,x_{n}^{\ast \ast })$, for any $x_{1},\ldots ,x_{n}\in \{0,1/2,1\}$.\medskip
\end{enumerate}

\noindent Since $^{\ast \ast }:\{0,1/2,1\}\rightarrow \{0,1/2,1\}$ is a
Heyting homomorphism, we have that (iii) implies (iv), and so (iii) implies
(i) and (ii). Suppose that (iv) holds and that $f$ is not a term function of 
$(\mathbf{3},\rightarrow ^{\mathbf{3}})$. We will arrive at a contradiction.
Note that (iv) implies that $\{0,1\}$ is closed under $f$. Since $f$ is not
a term function, the Baker-Pixley Theorem says that at least one of the
following subuniverses of $(\mathbf{3},\rightarrow ^{\mathbf{3}})\times (%
\mathbf{3},\rightarrow ^{\mathbf{3}})$ is not closed under $f\times f$,%
\begin{equation*}
\begin{array}{l}
S_{1}=\{(0,0),(1/2,1/2),(1,1/2),(1/2,1),(1,1)\}\text{,} \\ 
S_{2}=\{(0,0),(1/2,1),(1,1)\}\text{,} \\ 
S_{3}=\{(0,0),(1,1/2),(1,1)\}\text{.}%
\end{array}%
\end{equation*}%
(The other subuniverses of $(\mathbf{3},\rightarrow ^{\mathbf{3}})\times (%
\mathbf{3},\rightarrow ^{\mathbf{3}})$ are clearly closed under $f\times f$%
.) Suppose $S_{1}$ is not closed under $f\times f$. Note that $\mathbf{S}%
_{1} $ is generated by $\{(1,1/2),(1/2,1)\}$. So, we can suppose that $f$ is
binary, satisfies (iv) and $(f(1,1/2),f(1/2,1))\notin S_{1}$. If $%
(f(1,1/2),f(1/2,1))=(0,1/2)$, then%
\begin{equation*}
\begin{array}{lll}
(0,1) & = & (0^{\ast \ast },1/2^{\ast \ast }) \\ 
& = & (f(1^{\ast \ast },1/2^{\ast \ast }),f(1/2^{\ast \ast },1^{\ast \ast }))
\\ 
& = & (f(1,1),f(1,1))%
\end{array}%
\end{equation*}%
which is absurd. The other cases are similar. If either $S_{2}$ or $S_{3}$
is not closed under $f\times f$ we can arrive to a contradiction in a
similar manner.\medskip

\noindent \textbf{Claim 2.} Let $f:\{0,1/2,1\}^{n}\rightarrow \{0,1/2,1\}$.
If there is a formula $\varphi \in \left[ \exists \bigwedge \mathrm{At}(%
\mathcal{L})\right] $ which defines $f$ in $\mathbf{3}$ and $f$ is not a
term function of $\mathbf{3}$, then $\rightarrow ^{\mathbf{3}}$ is a term
function of $(\mathbf{3},f)$.\medskip

\noindent Proof. By Claim 1, $f$ is a term function of $(\mathbf{3}%
,\rightarrow ^{\mathbf{3}})$. By the Baker-Pixley Theorem, since $f$ is not
a term function of $\mathbf{3}$ we have that%
\begin{equation*}
S_{4}=\{(0,0),(1/2,1/2),(1/2,1),(1,1)\}
\end{equation*}%
or%
\begin{equation*}
S_{5}=\{(0,0),(1/2,1/2),(1,1/2),(1,1)\}
\end{equation*}%
is not closed under $f\times f$ (the other subuniverses are Heyting
subuniverses and hence they are closed under $f\times f$). But $\mathbf{S}%
_{4}$ and $\mathbf{S}_{5}$ are isomorphic and $f\times f$ is defined by $%
\varphi $ in $\mathbf{3}\times \mathbf{3}$ which implies that both $S_{4}$
and $S_{5}$ are not closed under $f\times f$. Thus every subalgebra of $(%
\mathbf{3},f)\times (\mathbf{3},f)$ is a Heyting subalgebra which by the
Baker-Pixley Theorem says that $\rightarrow ^{\mathbf{3}}$ is a term
operation of $(\mathbf{3},f)$.\medskip

\noindent \textbf{Claim 3.} If $\mathbf{L}$ is a Stone algebra which is a
subdirect product of a family of Stone algebras $\{\mathbf{L}_{i}:i\in I\}$,
and $\varphi \in \left[ \exists \bigwedge \mathrm{At}(\mathcal{L})\right] $
defines an $n$-ary function $f$ on $L$, then $\varphi $ defines a function $%
f_{i}$ on each $L_{i}$, and $f$ is the restriction of $(f_{i})_{i\in I}$ to $%
L^{n}$.\medskip

\noindent Proof. Since $\mathbf{L}_{i}$ is a homomorphic image of $\mathbf{L}
$, we have that $\mathbf{L}_{i}\vDash \exists z\varphi (\vec{x},z)$. Suppose 
$\varphi (\vec{x},z)=\exists \vec{w}\psi \left( \vec{w},\vec{x},z\right) $
with $\psi \in \left[ \bigwedge \mathrm{At}(\mathcal{L})\right] $. Since
every subquasivariety of the variety of Stone algebras is a variety and $%
\mathbf{L}\vDash \psi (\vec{w},\vec{x},z_{1})\wedge \psi (\vec{w},\vec{x}%
,z_{2})\rightarrow z_{1}=z_{2}$, we have that $\mathbf{L}_{i}\vDash \varphi (%
\vec{x},z_{1})\wedge \varphi (\vec{x},z_{2})\rightarrow z_{1}=z_{2}$, which
says that $\varphi $ defines a function on $L_{i}$.\medskip

\noindent \textbf{Claim 4.} If $\mathbf{A}\leq \mathbf{B}$, and $\varphi \in %
\left[ \exists \bigwedge \mathrm{At}(\mathcal{L})\right] $ defines $n$-ary
functions $f$ on $\mathbf{A}$ and $g$ on $\mathbf{B}$, then $f$ is equal to
the restriction of $g$ to $A^{n}$.\medskip

\noindent Proof. Trivial.\medskip

We are ready to prove (1). Suppose the Heyting implication exists in $%
\mathbf{L}$, and $(L,\vee ,\wedge ,\rightarrow ^{\mathbf{L}},0,1)$ is a
three valued Heyting algebra. Recall that $(\{0,1\},\max ,$ $\min
,\rightarrow ^{\mathbf{2}},0,1)$ and $(\{0,1/2,1\},\max ,\min ,\rightarrow ^{%
\mathbf{3}},0,1)$ are the only subdirectly irreducible three valued Heyting
algebras. Since $x^{\ast }=x\rightarrow ^{\mathbf{L}}0$ for every $x\in L$,
we can suppose that%
\begin{equation*}
(L,\vee ,\wedge ,^{\ast },\rightarrow ^{\mathbf{L}},0,1)\leq \prod \{\mathbf{%
L}_{u}:u\in I\cup J\}
\end{equation*}%
is a subdirect product, where $\mathbf{L}_{u}=(\mathbf{2},\rightarrow ^{%
\mathbf{2}})$ for $u\in I$, and $\mathbf{L}_{u}=(\mathbf{3},\rightarrow ^{%
\mathbf{3}})$ for $u\in J$. Suppose that $J\neq \emptyset $ and let $%
u_{0}\in J$. The case $J=\emptyset $ is left to the reader.

(a)$\Rightarrow $(b). This is clear.

(b)$\Rightarrow $(c). By Claim 3, $\varphi $ defines a function $f_{u}$ on
each $\mathbf{L}_{u}$ and $f$ is the restriction of $(f_{u})_{u\in I\cup J}$
to $L^{n}$. Note that $f_{u}=f_{v}$ whenever $u,v\in I$ or $u,v\in J$. By
Claim 4, for every $u\in I$, the function $f_{u}$ is the restriction of $%
f_{u_{0}}$ to $\{0,1\}^{n}$. By Claim 1, there is a $\{\vee ,\wedge ,^{\ast
},\rightarrow ,0,1\}$-term $p$ such that $f_{u_{0}}=p^{\mathbf{L}_{u_{0}}}$.
Note that $f_{u}=p^{\mathbf{L}_{u}}$, for every $u\in I\cup J$ and hence $%
f=p^{(L,\vee ,\wedge ,^{\ast },\rightarrow ^{\mathbf{L}},0,1)}$.

(c)$\Rightarrow $(a). Suppose $p$ is a $\{\vee ,\wedge ,^{\ast },\rightarrow
,0,1\}$-term such that $f=p^{(L,\vee ,\wedge ,^{\ast },\rightarrow ^{\mathbf{%
L}},0,1)}$. By Claim 1, $p^{\mathbf{L}_{u_{0}}}$ is definable in $\mathbf{L}%
_{u_{0}}=(\mathbf{3},\rightarrow ^{\mathbf{3}})$ by a formula $\varphi \in %
\left[ \bigwedge \mathrm{At}(\mathcal{L})\right] $. Note that, for every $%
u\in $ $I\cup J$, we have that $p^{\mathbf{L}_{u}}$ is defined in $\mathbf{L}%
_{u}$ by $\varphi $. But $f$ is the restriction of $(p^{\mathbf{L}%
_{u}})_{u\in I\cup J}$ to $L$, which implies that $f$ is defined by $\varphi 
$ in $\mathbf{L}$.

Next we prove (2). The implication (c)$\Rightarrow $(a) is immediate. In
fact, if an $\mathcal{L}$-term $t\left( \vec{x}\right) $ represents $f$ in $%
\mathbf{L}$, then the formula $z_{1}=t\left( \vec{x}\right) $ defines $f$ in 
$\mathbf{L}$. To show (b)$\Rightarrow $(c) we prove:\medskip

\noindent \textbf{Claim 5. }Assume $f:L^{n}\rightarrow L$ is not
representable by an $\mathcal{L}$-term in $\mathbf{L}$, and suppose $f$ is
defined in $\mathbf{L}$ by a formula $\varphi \in \left[ \exists \bigwedge 
\mathrm{At}(\mathcal{L})\right] $. Then the Heyting implication exists in $%
\mathbf{L}$, and $(L,\vee ,\wedge ,\rightarrow ^{\mathbf{L}},0,1)$ is a
three valued Heyting algebra.\medskip

To prove this claim we first note that since $\mathbf{2}$ and $\mathbf{3}$
are the only subdirectly irreducible Stone algebras we can suppose that%
\begin{equation*}
\mathbf{L}\leq \prod \{\mathbf{L}_{u}:u\in I\cup J\}
\end{equation*}%
is a subdirect product, where $\mathbf{L}_{u}=\mathbf{2}$ for $u\in I$, and $%
\mathbf{L}_{u}=\mathbf{3}$ for $u\in J$. By Claim 3, $\varphi $ defines a
function $f_{u}$ on each $\mathbf{L}_{u}$ and $f$ is the restriction of $%
(f_{u})_{u\in I\cup J}$ to $L$. Note that $f_{u}=f_{v}$ whenever $u,v\in I$
or $u,v\in J$. Also, by Claim 4, the function $f_{u}$ is the restriction of $%
f_{v}$ to $\{0,1\}$, whenever $u\in I$ and $v\in J$. If $J=\emptyset $, then
since $\mathbf{2}$ is primal, there is an $\mathcal{L}$-term $p$ which
represents $f_{u}$ in $\mathbf{2}$. But this is impossible since this
implies that $f$ is representable by $p$ in $\mathbf{L}$. So $J\neq
\emptyset $. Let $u_{0}\in J$. Since $f$ is not representable by an $%
\mathcal{L}$-term in $\mathbf{L}$, we have that $f_{u_{0}}$ is not
representable by an $\mathcal{L}$-term in $\mathbf{3}$. Thus Claim 2 implies
that $\rightarrow ^{\mathbf{3}}$ is a term function of $(\mathbf{3}%
,f_{u_{0}})$. Note that the same term witnesses that $\rightarrow ^{\mathbf{L%
}_{u}}$ is a term function of $(\mathbf{L}_{u},f_{u})$ for every $u\in I\cup
J$. But $\mathbf{L}$ is closed under $(f_{u})_{u\in I\cup J}$, and hence we
have that $\mathbf{L}$ is closed under $(\rightarrow ^{\mathbf{L}%
_{u}})_{u\in I\cup J}$. This says that the Heyting implication exists in $%
\mathbf{L}$, and that $(L,\vee ,\wedge ,\rightarrow ^{\mathbf{L}},0,1)$ is a
three valued Heyting algebra.
\end{proof}

Functions definable by a formula of the form $\dbigwedge p=q$ are called 
\emph{mono-algebraic}. They are studied in \cite{Ca-Va} using sheaf
representations.

\subsubsection{Primitive positive functions in De Morgan algebras}

Let $\mathcal{L}=\{\vee ,\wedge ,^{-},0,1\}$ be the language of De Morgan
algebras. Let $\mathbf{M}=(\{0,a,b,1\},\vee ,\wedge ,^{-},0,1)$ where $%
(\{0,a,b,1\},\vee ,\wedge ,0,1)$ is the two-atom Boolean lattice and $\bar{0}%
=1$, $\bar{1}=0$, $\bar{a}=a$ and $\bar{b}=b$. It is well known that $%
\mathbf{M}$ is a simple De Morgan algebra which generates the variety of all
De Morgan algebras \cite{ba-dw}. Primitive positive functions of $\mathbf{M}$
can be characterized as follows.

\begin{proposition}
\label{pp de morgan}For any function $f:M^{n}\rightarrow M$, the following
are equivalent

\begin{enumerate}
\item[(1)] There is a formula in $\left[ \bigwedge \mathrm{At}(\mathcal{L})%
\right] $ which defines $f$ in $\mathbf{M}$.

\item[(2)] There is a formula in $\left[ \exists \bigwedge \mathrm{At}(%
\mathcal{L})\right] $ which defines $f$ in $\mathbf{M}$.

\item[(3)] $f$ is a term operation of $(\mathbf{M},^{\circ })$, where $%
0^{\circ }=0$, $1^{\circ }=1$, $a^{\circ }=b$ and $b^{\circ }=a$.

\item[(4)] $f(x_{1},\ldots ,x_{n})^{\circ }=f(x_{1}^{\circ },\ldots
,x_{n}^{\circ })$, for any $x_{1},\ldots ,x_{n}\in M$.
\end{enumerate}
\end{proposition}

\begin{proof}
(1)$\Rightarrow $(2) This is trivial.

(2)$\Rightarrow $(3) We note that%
\begin{equation*}
^{\circ }:\{0,a,b,1\}\rightarrow \{0,a,b,1\}
\end{equation*}%
is an automorphism of $(\mathbf{M},^{\circ })$. Since $f$ is preserved by $%
^{\circ }$, assumption (2) implies that $\{0,1\}$ is closed under $f$. Also
note that%
\begin{equation*}
\overline{x}^{\circ }=\text{ Boolean complement of }x\text{,}
\end{equation*}%
which says that $(\mathbf{M},^{\circ })$ generates an arithmetical variety.
Since this algebra is simple and $\{0,1\}$ is its only proper subuniverse,
Fleischer's theorem says that the subuniverses of $(\mathbf{M},^{\circ
})\times (\mathbf{M},^{\circ })$ are: $\{0,1\}\times \{0,1\}$, $M\times M$, $%
\{(x,x^{\circ }):x\in M\}$, $\{(x,x):x\in M\}$ and $\{(0,0),(1,1)\}$. Each
of these is easily seen to be closed under $f\times f$, which by the
Baker-Pixley Theorem says that $f$ is a term function of $(\mathbf{M}%
,^{\circ })$.

(3)$\Rightarrow $(4) This is clear since $^{\circ }$ is an automorphism of $(%
\mathbf{M},^{\circ })$.

(4)$\Rightarrow $(1) Note that $\mathbf{M}$ has three non trivial inner
isomorphisms which are $^{\circ }$ and the restrictions of $^{\circ }$ to $%
\{0,a,1\}$ and $\{0,b,1\}$. Since the variety of De Morgan algebras is
congruence distributive, Corollary \ref{algebraicas para A con traduccion}
says that (1) holds.
\end{proof}

\section{Term interpolation}

Given a structure $\mathbf{A}$, an interesting --albeit often elusive--
problem is to provide a useful description of its term-operations. That is,
to give concise (semantical) conditions that characterize when a given
function $f:A^{n}\rightarrow A$ is a term-operation. This is beautifully
accomplished in the classical Baker-Pixley Theorem for the case in which $%
\mathbf{A}$ is finite and has a near-unanimity term \cite{ba-pi}.

A natural way to generalize this problem to classes of structures is as
follows. Given a class $\mathcal{K}$ of $\mathcal{L}$-structures and a map $%
\mathbf{A}\rightarrow f^{\mathbf{A}}$ which assigns to each $\mathbf{A}\in 
\mathcal{K}$ an $n$-ary operation $f^{\mathbf{A}}:A^{n}$ $\rightarrow A$,
provide conditions that guarantee the existence of a term $t$ such that $t^{%
\mathbf{A}}=f^{\mathbf{A}}$ on every $\mathbf{A}$ in $\mathcal{K}$. We
address this problem for classes in the current section, and obtain some
interesting results including generalizations of the aforementioned
Baker-Pixley Theorem and for Pixley's theorem characterizing the
term-operations of quasiprimal algebras \cite{Pixley}.

Another avenue of generalization we considered are functions that are
interpolated by a finite number of terms. This is also looked at in the
setting of classes.

\subsection*{Term-valued functions by cases}

Let $f\in \mathcal{L}$ be a function symbol, and let $\mathcal{K}$ be a
class of $\mathcal{L}$-structures. Given $\mathcal{L}$-terms $t_{1}(\vec{x}%
),\dots ,t_{k}(\vec{x})$ and first order $\mathcal{L}$-formulas $\varphi
_{1}(\vec{x}),\dots ,\varphi _{k}(\vec{x})$, we write%
\begin{equation*}
f=\left. t_{1}\right\vert _{\varphi _{1}}\cup \ldots \cup \left.
t_{k}\right\vert _{\varphi _{k}}\text{ in }\mathcal{K}
\end{equation*}%
to express that $\mathcal{K}\vDash \varphi _{1}(\vec{x})\vee \dots \vee
\varphi _{k}(\vec{x})$ and%
\begin{equation*}
f^{\mathbf{A}}(\vec{a})=\left\{ 
\begin{array}{cc}
t_{1}^{\mathbf{A}}(\vec{a}) & \text{if }\mathbf{A}\vDash \varphi _{1}(\vec{a}%
) \\ 
\vdots & \vdots \\ 
t_{k}^{\mathbf{A}}(\vec{a}) & \text{if }\mathbf{A}\vDash \varphi _{k}(\vec{a}%
)%
\end{array}%
\right.
\end{equation*}%
for all $\vec{a}\in A^{n}$ and $\mathbf{A}\in \mathcal{K}$. (Note that as $%
f^{\mathbf{A}}$ is a function the definition by cases is not ambiguous.) We
say that a term $t(\vec{x})$ \emph{represents }$f$ in $\mathcal{K}$ if $f^{%
\mathbf{A}}(\vec{a})=t^{\mathbf{A}}(\vec{a})$, for all $\mathbf{A}\in 
\mathcal{K}$ and $\vec{a}\in A^{n}$.

With the help of results from previous sections it is possible to
characterize when $f=\left. t_{1}\right\vert _{\varphi _{1}}\cup \ldots \cup
\left. t_{k}\right\vert _{\varphi _{k}}$, with the $t_{i}$'s not involving $%
f $ and a fixed format for the $\varphi _{i}$'s.

\begin{theorem}
\label{term valued with open cases}Let $\mathcal{L}\subseteq \mathcal{L}%
^{\prime }$ be first order languages and let $f\in \mathcal{L}^{\prime }-%
\mathcal{L}$ be an $n$-ary function symbol. Let $\mathcal{K}$ be a class of $%
\mathcal{L}^{\prime }$-structures. The following are equivalent:

\begin{enumerate}
\item[(1)] $f=\left. t_{1}\right\vert _{\varphi _{1}}\cup \ldots \cup \left.
t_{k}\right\vert _{\varphi _{k}}$ in $\mathcal{K}$, with each $t_{i}$ an $%
\mathcal{L}$-term and each $\varphi _{i}$ in $\mathrm{Op}(\mathcal{L})$.

\item[(2)] The following conditions hold:

\begin{enumerate}
\item[(a)] For all $\mathbf{A}\in \mathbb{P}_{u}(\mathcal{K})$ and all $%
\mathbf{S}\leq \mathbf{A}_{\mathcal{L}}$, we have that $S$ is closed under $%
f^{\mathbf{A}}$.

\item[(b)] For all $\mathbf{A},\mathbf{B}\in \mathbb{P}_{u}(\mathcal{K})$,
all $\mathbf{A}_{0}\leq \mathbf{A}_{\mathcal{L}}$, $\mathbf{B}_{0}\leq 
\mathbf{B}_{\mathcal{L}}$, and all isomorphisms $\sigma :\mathbf{A}%
_{0}\rightarrow \mathbf{B}_{0}$, we have that $\sigma :(\mathbf{A}_{0},f^{%
\mathbf{A}}|_{A_{0}})\rightarrow (\mathbf{B}_{0},f^{\mathbf{B}}|_{B_{0}})$
is an isomorphism.
\end{enumerate}
\end{enumerate}

\noindent Moreover, if $\mathcal{K}_{\mathcal{L}}$ has finitely many
isomorphism types of $(n+1)$-generated substructures and each one is finite,
then we can remove the operator $\mathbb{P}_{u}$ from (a) and (b).
\end{theorem}

\begin{proof}
(1)$\Rightarrow $(2). If $f=\left. t_{1}\right\vert _{\varphi _{1}}\cup
\ldots \cup \left. t_{k}\right\vert _{\varphi _{k}}$ in $\mathcal{K}$, with
each $t_{i}$ an $\mathcal{L}$-term and each $\varphi _{i}$ in $\mathrm{Op}(%
\mathcal{L})$, then $f=\left. t_{1}\right\vert _{\varphi _{1}}\cup \ldots
\cup \left. t_{k}\right\vert _{\varphi _{k}}$ in $\mathbb{P}_{u}(\mathcal{K}%
) $. Now (a) and (b) are routine verifications.

(2)$\Rightarrow $(1). We first note that (a) implies%
\begin{equation*}
\mathbb{P}_{u}(\mathcal{K})\vDash \dbigvee_{t(\vec{x})\text{ an }\mathcal{L}%
\text{-term }}f(\vec{x})=t(\vec{x})\text{,}
\end{equation*}%
which by compactness says that there are $\mathcal{L}$-terms $t_{1}(\vec{x}%
),\dots ,t_{k}(\vec{x})$ such that%
\begin{equation*}
\mathbb{P}_{u}(\mathcal{K})\vDash f(\vec{x})=t_{1}(\vec{x})\vee \dots \vee f(%
\vec{x})=t_{k}(\vec{x})\text{.}
\end{equation*}%
Since (b) holds, Theorem \ref{(positive) open} implies that there exists a
formula $\varphi \in \mathrm{Op}(\mathcal{L})$ which defines $f$ in $%
\mathcal{K}$. It is clear that for any $\mathbf{A}\in \mathcal{K}$ we have%
\begin{equation*}
\mathcal{K}\vDash \varphi (\vec{x},t_{1}(\vec{x}))\vee \dots \vee \varphi (%
\vec{x},t_{k}(\vec{x}))
\end{equation*}%
and%
\begin{equation*}
f^{\mathbf{A}}(\vec{a})=\left\{ 
\begin{array}{cc}
t_{1}^{\mathbf{A}}(\vec{a}) & \text{if }\mathbf{A}\vDash \varphi (\vec{a}%
,t_{1}(\vec{a})) \\ 
\vdots & \vdots \\ 
t_{k}^{\mathbf{A}}(\vec{a}) & \text{if }\mathbf{A}\vDash \varphi (\vec{a}%
,t_{k}(\vec{a}))%
\end{array}%
\right.
\end{equation*}%
for all $\vec{a}\in A^{n}$. So we have proved that $f=\left.
t_{1}\right\vert _{\varphi _{1}}\cup \ \ldots \ \cup \left. t_{k}\right\vert
_{\varphi _{k}}$ in $\mathcal{K}$, where $\varphi _{i}(\vec{x})=\varphi (%
\vec{x},t_{i}(\vec{x}))$, $i=1,\ldots ,k$.

If $\mathcal{K}_{\mathcal{L}}$ has finitely many isomorphism types of $(n+1)$%
-generated substructures and each one is finite, then we note that there are 
$\mathcal{L}$-terms $t_{1}(\vec{x}),\dots ,t_{k}(\vec{x})$ such that%
\begin{equation*}
\mathcal{K}\vDash f(\vec{x})=t_{1}(\vec{x})\vee \dots \vee f(\vec{x})=t_{k}(%
\vec{x})
\end{equation*}%
and the proof can be continued in the same manner as above.
\end{proof}

\begin{corollary}
Let $\mathcal{K}$ be any class of $\mathcal{L}$-algebras contained in a
locally finite variety. Suppose $\mathbf{A}\rightarrow f^{\mathbf{A}}$ is a
map which assigns to each $\mathbf{A}\in \mathcal{K}$ an $n$-ary operation $%
f^{\mathbf{A}}:A^{n}$ $\rightarrow A$. The following are equivalent:

\begin{enumerate}
\item[(1)] $f=\left. t_{1}\right\vert _{\varphi _{1}}\cup \ldots \cup \left.
t_{k}\right\vert _{\varphi _{k}}$ in $\mathcal{K}$, with each $t_{i}$ an $%
\mathcal{L}$-term and each $\varphi _{i}$ in $\mathrm{Op}(\mathcal{L})$.

\item[(2)] The following conditions hold:

\begin{enumerate}
\item[(a)] If $\mathbf{S}\leq \mathbf{A}\in \mathcal{K}$, then $S$ is closed
under $f^{\mathbf{A}}$.

\item[(b)] If $\mathbf{A}_{0}\leq \mathbf{A}\in \mathcal{K}$, $\mathbf{B}%
_{0}\leq \mathbf{B}\in \mathcal{K}$ and $\sigma :\mathbf{A}_{0}\rightarrow 
\mathbf{B}_{0}$ is an isomorphism, then $\sigma :(\mathbf{A}_{0},f^{\mathbf{A%
}})\rightarrow (\mathbf{B}_{0},f^{\mathbf{B}})$ is an isomorphism.
\end{enumerate}
\end{enumerate}
\end{corollary}

Using Theorem \ref{term valued with open cases} and its proof as a blueprint
it is easy to produce analogous results for other families of formulas. For
example here is the positive open case.

\begin{theorem}
\label{term valued with positive cases}Let $\mathcal{L}\subseteq \mathcal{L}%
^{\prime }$ be first order languages and let $f\in \mathcal{L}^{\prime }-%
\mathcal{L}$ be an $n$-ary function symbol. Let $\mathcal{K}$ be a class of $%
\mathcal{L}^{\prime }$-structures. The following are equivalent:

\begin{enumerate}
\item[(1)] $f=\left. t_{1}\right\vert _{\varphi _{1}}\cup \ldots \cup \left.
t_{k}\right\vert _{\varphi _{k}}$ in $\mathcal{K}$, with each $t_{i}$ an $%
\mathcal{L}$-term and each $\varphi _{i}$ in $\left[ \bigvee \bigwedge 
\mathrm{At}(\mathcal{L})\right] $.

\item[(2)] $f=\left. t_{1}\right\vert _{\varphi _{1}}\cup \ldots \cup \left.
t_{k}\right\vert _{\varphi _{k}}$ in $\mathcal{K}$, with each $t_{i}$ an $%
\mathcal{L}$-term and each $\varphi _{i}$ in $\left[ \bigwedge \mathrm{At}(%
\mathcal{L})\right] $.

\item[(3)] The following conditions hold:

\begin{enumerate}
\item[(a)] For all $\mathbf{A}\in \mathbb{P}_{u}(\mathcal{K})$ and all $%
\mathbf{S}\leq \mathbf{A}_{\mathcal{L}}$, we have that $S$ is closed under $%
f^{\mathbf{A}}$.

\item[(b)] For all $\mathbf{A},\mathbf{B}\in \mathbb{P}_{u}(\mathcal{K})$,
all $\mathbf{A}_{0}\leq \mathbf{A}_{\mathcal{L}}$, $\mathbf{B}_{0}\leq 
\mathbf{B}_{\mathcal{L}}$, and all homomorphisms $\sigma :\mathbf{A}%
_{0}\rightarrow \mathbf{B}_{0}$, we have that $\sigma :(\mathbf{A}_{0},f^{%
\mathbf{A}})\rightarrow (\mathbf{B}_{0},f^{\mathbf{B}})$ is a homomorphism.
\end{enumerate}
\end{enumerate}

\noindent Moreover, if $\mathcal{K}_{\mathcal{L}}$ has finitely many
isomorphism types of $(n+1)$-generated substructures and each one is finite,
then we can remove the operator $\mathbb{P}_{u}$ from (a) and (b).
\end{theorem}

\subsection*{Pixley's theorem for classes}

Recall that the \emph{ternary discriminator} on the set $\mathbf{A}$ is the
function%
\begin{equation*}
d^{\mathbf{A}}(x,y,z)=\left\{ 
\begin{tabular}{ll}
$z$ & \ \ if $x=y$, \\ 
$x$ & \ \ if $x\neq y$.%
\end{tabular}%
\ \right.
\end{equation*}%
An algebra $\mathbf{A}$ is called \textit{quasiprimal} if it is finite and
has the discriminator as a term function.

A well known result of A. Pixley \cite{Pixley} characterizes quasiprimal
algebras as those finite algebras in which every function preserving the
inner isomorphisms is a term function. Of course the ternary discriminator
preserves the inner isomorphisms and hence one direction of the theorem is
trivial. The following theorem generalizes the non trivial direction.

\begin{theorem}
\label{generalizacion imp no trivial de pixley}Let $\mathcal{L}\subseteq 
\mathcal{L}^{\prime }$ be first order languages without relation symbols and
let $f\in \mathcal{L}^{\prime }-\mathcal{L}$ be an $n$-ary function symbol.
Let $\mathcal{K}$ be a class of $\mathcal{L}^{\prime }$-algebras such that
there is an $\mathcal{L}$-term representing the ternary discriminator in
each member of $\mathcal{K}$. Suppose that

\begin{enumerate}
\item[(a)] If $\mathbf{A}\in \mathbb{P}_{u}(\mathcal{K})$ and $\mathbf{S}%
\leq \mathbf{A}_{\mathcal{L}}$, then $S$ is closed under $f^{\mathbf{A}}$.

\item[(b)] If $\mathbf{A},\mathbf{B}\in \mathbb{P}_{u}(\mathcal{K})$, $%
\mathbf{A}_{0}\leq \mathbf{A}_{\mathcal{L}}$, $\mathbf{B}_{0}\leq \mathbf{B}%
_{\mathcal{L}}$ and $\sigma :\mathbf{A}_{0}\rightarrow \mathbf{B}_{0}$ is an
isomorphism, then $\sigma :(\mathbf{A}_{0},f^{\mathbf{A}})\rightarrow (%
\mathbf{B}_{0},f^{\mathbf{B}})$ is an isomorphism.
\end{enumerate}

\noindent Then $f$ is representable by an $\mathcal{L}$-term in $\mathcal{K}$%
. Moreover, if $\mathcal{K}_{\mathcal{L}}$ has finitely many isomorphism
types of $(n+1)$-generated subalgebras and each one is finite, then we can
remove the operator $\mathbb{P}_{u}$ from (a) and (b).
\end{theorem}

\begin{proof}
By Theorem \ref{term valued with open cases} we have that $f=\left.
t_{1}\right\vert _{\varphi _{1}}\cup \ldots \cup \left. t_{k}\right\vert
_{\varphi _{k}}$ in $\mathcal{K}$, with each $t_{i}$ an $\mathcal{L}$-term
and each $\varphi _{i}$ in $\mathrm{Op}(\mathcal{L})$. We shall prove that $%
f $ is representable by an $\mathcal{L}$-term in $\mathcal{K}$. Of course,
if $k=1$, the theorem follows. Suppose $k>1$. We show that we can reduce $k$.

Let $t\left( x,y,z\right) $ be an $\mathcal{L}$-term representing the
discriminator in $\mathcal{K}$. Then the $\mathcal{L}$-term%
\begin{equation*}
D(x,y,z,w)=t\left( t\left( x,y,z\right) ,t\left( x,y,w\right) ,w\right)
\end{equation*}%
represents the \textit{quaternary discriminator} in $\mathcal{K}$, that is,
for every $a,b,c,d\in A$, with $\mathbf{A}\in \mathcal{K}$,%
\begin{equation*}
D^{\mathbf{A}}(a,b,c,d)=\left\{ 
\begin{tabular}{ll}
$c$ & $a=b$ \\ 
$d$ & $a\neq b$.%
\end{tabular}%
\right.
\end{equation*}%
Having a discriminator term for $\mathcal{K}$ also provides the following
translation property (see \cite{we}):

\begin{itemize}
\item For every open $\mathcal{L}$-formula $\varphi (\vec{x})$ there exist $%
\mathcal{L}$-terms $p(\vec{x})$ and $q(\vec{x})$ such that either $\mathcal{K%
}\vDash \varphi (\vec{x})\leftrightarrow p(\vec{x})=q(\vec{x})$ or $\mathcal{%
K}\vDash \varphi (\vec{x})\leftrightarrow p(\vec{x})\neq q(\vec{x})$.
\end{itemize}

Now, suppose for example that%
\begin{equation*}
\begin{tabular}{l}
$\mathcal{K}\vDash \varphi _{1}(\vec{x})\leftrightarrow p_{1}(\vec{x})\neq
q_{1}(\vec{x})$, \\ 
$\mathcal{K}\vDash \varphi _{2}(\vec{x})\leftrightarrow p_{2}(\vec{x})\neq
q_{2}(\vec{x})$,%
\end{tabular}%
\end{equation*}%
for some $\mathcal{L}$-terms $p_{i}(\vec{x})$ and $q_{i}(\vec{x})$. Then%
\begin{equation*}
f^{\mathbf{A}}(\overline{a})=\left\{ 
\begin{array}{ll}
D(p_{1},q_{1},t_{2},t_{1})^{\mathbf{A}}(\vec{a}) & \ \ \ \text{if }\mathbf{A}%
\vDash \varphi (\vec{a}) \\ 
t_{3}^{\mathbf{A}}(\vec{a}) & \ \ \ \text{if }\mathbf{A}\vDash \varphi _{3}(%
\vec{a}) \\ 
\ \ \ \vdots & \ \ \ \ \ \ \ \ \ \ \ \vdots \\ 
t_{k}^{\mathbf{A}}(\vec{a}) & \ \ \ \text{if }\mathbf{A}\vDash \varphi _{k}(%
\vec{a})%
\end{array}%
\right.
\end{equation*}%
where $\varphi (\vec{x})=p_{1}(\vec{x})\neq q_{1}(\vec{x})\vee \left( p_{1}(%
\vec{x})=q_{1}(\vec{x})\wedge p_{2}(\vec{x})\neq q_{2}(\vec{x})\right) $.
The other cases are similar.
\end{proof}

For the locally finite case Pixley's theorem can be generalized as follows.

\begin{theorem}
\label{Pixley}Let $\mathcal{K}$ be a class of $\mathcal{L}$-algebras
contained in a locally finite variety. The following are equivalent:

\begin{enumerate}
\item[(1)] There is an $\mathcal{L}$-term representing the ternary
discriminator on each member of $\mathcal{K}$.

\item[(2)] Assume $\mathbf{A}\rightarrow f^{\mathbf{A}}$ is a map which
assigns to each $\mathbf{A}\in \mathcal{K}$ an $n$-ary operation $f^{\mathbf{%
A}}:A^{n}$ $\rightarrow A$ in such a manner that:

\begin{enumerate}
\item[(a)] for all $\mathbf{A}\in \mathcal{K}$ and all $\mathbf{S}\leq 
\mathbf{A}$ we have that $S$ is closed under $f^{\mathbf{A}}$, and

\item[(b)] for all $\mathbf{A},\mathbf{B}\in \mathcal{K}$, all $\mathbf{A}%
_{0}\leq \mathbf{A}$, $\mathbf{B}_{0}\leq \mathbf{B}$, and every isomorphism 
$\sigma :\mathbf{A}_{0}\rightarrow \mathbf{B}_{0}$ we have that $\sigma :(%
\mathbf{A}_{0},f^{\mathbf{A}})\rightarrow (\mathbf{B}_{0},f^{\mathbf{B}})$
is an isomorphism.
\end{enumerate}

\noindent Then $f$ is representable by an $\mathcal{L}$-term in $\mathcal{K}$%
.
\end{enumerate}
\end{theorem}

\begin{proof}
(1)$\Rightarrow $(2). This follows from Theorem \ref{generalizacion imp no
trivial de pixley} applied to the class $\{(\mathbf{A},f^{\mathbf{A}}\dot{)}:%
\mathbf{A}\in \mathcal{K\}}$.

(2)$\Rightarrow $(1). It is clear that the map $\mathbf{A}\rightarrow d^{%
\mathbf{A}}$ which assigns to each $\mathbf{A}\in \mathcal{K}$ the ternary
discriminator on $\mathbf{A}$, satisfies (a) and (b) of (2).
\end{proof}

\subsection*{Baker-Pixley's theorem for classes}

Let $\mathcal{K}$ be a class of $\mathcal{L}$-algebras. An $\mathcal{L}$%
-term $M\left( x,y,z\right) $ is called a \textit{majority term for} $%
\mathcal{K}$ if the following identities hold in $\mathcal{K}$%
\begin{equation*}
M(x,x,y)\approx M(x,y,x)\approx M(y,x,x)\approx x\text{.}
\end{equation*}%
Next, we shall give a generalization of the well known theorem of Baker and
Pixley \cite{ba-pi} on the existence of terms representing functions in
finite algebras with a majority term. First a lemma.

\begin{lemma}
\label{existencia de terminos en general}Let $\mathcal{K}$ be a class of $%
\mathcal{L}$-algebras contained in a locally finite variety. Let $\mathbf{A}%
\rightarrow f^{\mathbf{A}}$ be a map which assigns to each $\mathbf{A}\in 
\mathcal{K}$ an $n$-ary operation $f^{\mathbf{A}}:A^{n}$ $\rightarrow A$.
Suppose that for any $m\in \mathbb{N}$, $\mathbf{A}_{1},\dots ,\mathbf{A}%
_{m}\in \mathcal{K}$ and $\mathbf{S}\leq \mathbf{A}_{1}\times \ldots \times 
\mathbf{A}_{m}$, we have that $S$ is closed under $f^{\mathbf{A}_{1}}\times
\ldots \times f^{\mathbf{A}_{m}}$. Then $f$ is representable by an $\mathcal{%
L}$-term in $\mathcal{K}$.
\end{lemma}

\begin{proof}
Let $t_{1}(\vec{x}),\dots ,t_{k}(\vec{x})$ be $\mathcal{L}$-terms such that
for every $\mathbf{A}$ in the variety generated by $\mathcal{K}$ and every $%
\vec{a}\in A^{n}$ we have that the subalgebra of $\mathbf{A}$ generated by $%
a_{1},\ldots ,a_{n}$ has universe $\{t_{1}^{\mathbf{A}}(\vec{a}),\ldots
,t_{k}^{\mathbf{A}}(\vec{a})\}$. We prove that $f$ is representable by $%
t_{i} $ in $\mathcal{K}$, for some $i$. Suppose to the contrary that for
each $i$ there are $\mathbf{A}_{i}\in \mathcal{K}$ and $\vec{a}%
^{i}=(a_{1}^{i},\ldots ,a_{n}^{i})\in A_{i}^{n}$ such that $f^{\mathbf{A}%
_{i}}(\vec{a}^{i})\neq t_{i}^{\mathbf{A}_{i}}(\vec{a}^{i})$. Let $\mathbf{S}$
be the subalgebra of $\mathbf{A}_{1}\times \ldots \times \mathbf{A}_{k}$
generated by $\{p_{j}:j=1,\ldots ,n\}$, where $p_{j}=(a_{j}^{1},a_{j}^{2},%
\ldots ,a_{j}^{k})$. Since $S$ is closed under $f^{\mathbf{A}_{1}}\times
\ldots \times f^{\mathbf{A}_{k}}$ we have that%
\begin{equation*}
f^{\mathbf{A}_{1}}\times \ldots \times f^{\mathbf{A}_{k}}(p_{1},\ldots
,p_{n})=(f^{\mathbf{A}_{1}}(\vec{a}^{1}),\ldots ,f^{\mathbf{A}_{k}}(\vec{a}%
^{k}))\in S\text{.}
\end{equation*}%
Thus there is $i$ such that%
\begin{equation*}
(f^{\mathbf{A}_{1}}(\vec{a}^{1}),\ldots ,f^{\mathbf{A}_{k}}(\vec{a}%
^{k}))=t_{i}^{\mathbf{A}_{1}\times \ldots \times \mathbf{A}%
_{m}}(p_{1},\ldots ,p_{n})\text{,}
\end{equation*}%
which produces%
\begin{equation*}
(f^{\mathbf{A}_{1}}(\vec{a}^{1}),\ldots ,f^{\mathbf{A}_{k}}(\vec{a}%
^{k}))=(t_{i}^{\mathbf{A}_{1}}(\vec{a}^{1}),\ldots ,t_{i}^{\mathbf{A}_{k}}(%
\vec{a}^{k}))\text{.}
\end{equation*}%
In particular we have that $f^{\mathbf{A}_{i}}(\vec{a}^{i})=t_{i}^{\mathbf{A}%
_{i}}(\vec{a}^{i})$, a contradiction.
\end{proof}

\begin{theorem}
\label{Baker-Pixley para clases caso localmente finitas}Let $\mathcal{K}$ be
a class of algebras contained in a locally finite variety and suppose that $%
\mathcal{K}$ has a majority term. Let $\mathbf{A}\rightarrow f^{\mathbf{A}}$
be a map which assigns to each $\mathbf{A}\in \mathcal{K}$ an $n$-ary
operation $f^{\mathbf{A}}:A^{n}$ $\rightarrow A$. Assume that for all $%
\mathbf{A},\mathbf{B}\in \mathcal{K}$ and every $\mathbf{S}\leq \mathbf{A}%
\times \mathbf{B}$ we have that $S$ is closed under $f^{\mathbf{A}}\times f^{%
\mathbf{B}}$. Then $f$ is representable by a term in $\mathcal{K}$.
\end{theorem}

\begin{proof}
First we show that:

\begin{enumerate}
\item[(i)] If $\mathbf{A},\mathbf{B}\in \mathcal{K}$, $\vec{a}\in A^{n}$ and 
$\vec{b}\in B^{n}$, then there\ is\ a\ term $t(\vec{x})$\ satisfying\ $t^{%
\mathbf{A}}(\vec{a})=f^{\mathbf{A}}(\vec{a})$ and $t^{\mathbf{B}}(\vec{b}%
)=f^{\mathbf{B}}(\vec{b})$.
\end{enumerate}

\noindent Let $\mathbf{S}$ be the subalgebra of $\mathbf{A}\times \mathbf{B}$
generated by $\{(a_{1},b_{1}),\ldots ,(a_{n},b_{n})\}$. Since $S$ is closed
under $f^{\mathbf{A}}\times f^{\mathbf{B}}$ we have that%
\begin{equation*}
f^{\mathbf{A}}\times f^{\mathbf{B}}((a_{1},b_{1}),\ldots ,(a_{n},b_{n}))=(f^{%
\mathbf{A}}(\vec{a}),f^{\mathbf{B}}(\vec{b}))\in S\text{.}
\end{equation*}%
Thus there is a term $t(\vec{x})$ such that%
\begin{equation*}
(f^{\mathbf{A}}(\vec{a}),f^{\mathbf{B}}(\vec{b}))=t^{\mathbf{A}\times 
\mathbf{B}}((a_{1},b_{1}),\ldots ,(a_{n},b_{n}))=(t^{\mathbf{A}}(\vec{a}),t^{%
\mathbf{B}}(\vec{b}))\text{.}
\end{equation*}%
Next we prove by induction in $m$ that:

\begin{enumerate}
\item[(I$_{m}$)] If $\mathbf{A}_{1},\dots ,\mathbf{A}_{m}\in \mathcal{K}$
and\ $\vec{a}_{j}\in A_{j}^{n}$, for $j=1,\ldots ,m$,\ then there\ is\ a\
term $t(\vec{x})$\ satisfying\ $t^{\mathbf{A}_{j}}(\vec{a}_{j})=f^{\mathbf{A}%
_{j}}(\vec{a}_{j})$, for $j=1,\ldots ,m$.
\end{enumerate}

\noindent By (i) we have that (I$_{m}$) holds for $m=1,2$. Fix $m\geq 3$, $%
\mathbf{A}_{1},\dots ,\mathbf{A}_{m}\in \mathcal{K}$ and\ $\vec{a}_{j}\in
A_{j}^{n}$, for $j=1,\ldots ,m$. By inductive hypothesis there are terms $%
t_{1}$, $t_{2}$ and $t_{3}$ satisfying%
\begin{equation*}
\begin{tabular}{l}
$t_{1}^{\mathbf{A}_{j}}(\vec{a}_{j})=f^{\mathbf{A}_{j}}(\vec{a}_{j})$, for
all $j\neq 1$, $1\leq j\leq m$, \\ 
$t_{2}^{\mathbf{A}_{j}}(\vec{a}_{j})=f^{\mathbf{A}_{j}}(\vec{a}_{j})$, for
all $j\neq 2$, $1\leq j\leq m$, \\ 
$t_{3}^{\mathbf{A}_{j}}(\vec{a}_{j})=f^{\mathbf{A}_{j}}(\vec{a}_{j})$, for
all $j\neq 3$, $1\leq j\leq m$.%
\end{tabular}%
\end{equation*}%
It is easy to check that $t=M(t_{1},t_{2},t_{3})$ satisfies%
\begin{equation*}
t^{\mathbf{A}_{j}}(\vec{a}_{j})=f^{\mathbf{A}_{j}}(\vec{a}_{j})\text{, for }%
j=1,\ldots ,m\text{.}
\end{equation*}%
We observe that the fact that (I$_{m}$) holds for every $m\geq 1$ implies
that the hypothesis of Lemma \ref{existencia de terminos en general} holds
and hence $f$ is representable by a term in $\mathcal{K}$.
\end{proof}

We conclude the section with another term-interpolation result in the spirit
of the Baker-Pixley Theorem -- in this case, for classes contained in
arithmetical varieties having a universal class of finitely subdirectly
irreducibles. There are plenty of well-known examples of this kind of
varieties, we list a few: f-rings, vector groups, MV-algebras, Heyting
algebras, discriminator varieties, etc.

In our proof we use the notion of a \emph{global subdirect product}, which
is a special kind of subdirect product, tight enough so that significant
information can be obtained from the properties of the factors. We do not
provide the definition here but refer the reader to \cite{kr-cl}.

We write $\mathcal{V}_{FSI}$ to denote the class of finitely subdirectly
irreducible members of a variety $\mathcal{V}$.

\begin{lemma}
\label{existencia de terminos en variedades aritmeticas}Let $\mathcal{V}$ be
an arithmetical variety of $\mathcal{L}$-algebras and suppose that $\mathcal{%
V}_{FSI}$ is universal. If $\mathcal{V}_{FSI}\vDash \forall \vec{x}\exists
!z\ \varphi (\vec{x},z)$, where $\varphi \in \left[ \bigwedge \mathrm{At}(%
\mathcal{L})\right] $, then there exists an $\mathcal{L}$-term $t(\vec{x})$
such that $\mathcal{V}\vDash \forall \vec{x}\ \varphi (\vec{x},t(\vec{x}))$.
\end{lemma}

\begin{proof}
By \cite{gr-va} every algebra of $\mathcal{V}$ is isomorphic to a global
subdirect product whose factors are finitely subdirectly irreducible. Since
global subdirect products preserve $(\forall \exists !\bigwedge p=q)$%
-sentences (see \cite{vo}), we have that $\mathcal{V}\vDash \forall \vec{x}%
\exists !z\ \varphi (\vec{x},z)$. Let $\mathbf{F}$ be the algebra of $%
\mathcal{V}$ freely generated by $x_{1},\ldots ,x_{n}$. Since $\mathbf{F}%
\vDash \exists !z\ \varphi (\vec{x},z)$, there exists a term $t(\vec{x})$
such that $\mathbf{F}\vDash \varphi (\vec{x},t(\vec{x}))$. It is easy to
check that $\mathcal{V}\vDash \forall \vec{x}\ \varphi (\vec{x},t(\vec{x}))$.
\end{proof}

For a class of $\mathcal{L}$-algebras $\mathcal{K}$ let $\mathbb{V}(\mathcal{%
K})$ denote the $\emph{variety}$\emph{\ generated by }$\mathcal{K}$.

\begin{theorem}
\label{Baker-Pixley para aritmeticas}Let $\mathcal{L}\subseteq \mathcal{L}%
^{\prime }$ be first order languages without relation symbols and let $f\in 
\mathcal{L}^{\prime }-\mathcal{L}$ be an $n$-ary function symbol. Let $%
\mathcal{K}$ be a class of $\mathcal{L}^{\prime }$-algebras satisfying:

\begin{enumerate}
\item[(a)] $\mathbb{V}(\mathcal{K}_{\mathcal{L}})$ is arithmetical and $%
\mathbb{V}(\mathcal{K}_{\mathcal{L}})_{FSI}$ is a universal class.

\item[(b)] If $\mathbf{A},\mathbf{B}\in \mathbb{P}_{u}(\mathcal{K})$ and $%
\mathbf{S}\leq \mathbf{A}_{\mathcal{L}}\times \mathbf{B}_{\mathcal{L}}$,
then $S$ is closed under $f^{\mathbf{A}}\times f^{\mathbf{B}}$.
\end{enumerate}

\noindent Then $f$ is representable by an $\mathcal{L}$-term in $\mathcal{K}$%
.
\end{theorem}

\begin{proof}
W.l.o.g. we can suppose that $\mathcal{L}^{\prime }=\mathcal{L}\cup \{f\}$.
As in the proof of Theorem \ref{Baker-Pixley para clases caso localmente
finitas} we can see that given $\mathbf{A},\mathbf{B}\in \mathbb{SP}_{u}(%
\mathcal{K})$, $\vec{a}\in A^{n}$ and $\vec{b}\in B^{n}$, there is an $%
\mathcal{L}$-term $t(\vec{x})$ satisfying\emph{\ }$t^{\mathbf{A}}(\vec{a}%
)=f^{\mathbf{A}}(\vec{a})$ and $t^{\mathbf{B}}(\vec{b})=f^{\mathbf{B}}(\vec{b%
})$. We establish this property in a wider class.

\begin{enumerate}
\item[(i)] If $\mathbf{A},\mathbf{B}\in \mathbb{HSP}_{u}(\mathcal{K})$, $%
\vec{a}\in A^{n}$ and $\vec{b}\in B^{n}$, then there is an $\mathcal{L}$%
-term $t(\vec{x})$ such that $t^{\mathbf{A}}(\vec{a})=f^{\mathbf{A}}(\vec{a}%
) $ and $t^{\mathbf{B}}(\vec{b})=f^{\mathbf{B}}(\vec{b})$.
\end{enumerate}

\noindent Take $\mathbf{A},\mathbf{B}\in \mathbb{HSP}_{u}(\mathcal{K})$, and
fix $\vec{a}\in A^{n}$ and $\vec{b}\in B^{n}$. There are $\mathbf{A}_{1}\in 
\mathbb{S}\mathbb{P}_{u}(\mathcal{K})$, $\mathbf{B}_{1}\in \mathbb{SP}_{u}(%
\mathcal{K})$ and onto homomorphisms $F:\mathbf{A}_{1}\rightarrow \mathbf{A}$
and $G:\mathbf{B}_{1}\rightarrow \mathbf{B}$. Let $\vec{c}\in A_{1}^{n}$ and 
$\vec{d}\in B_{1}^{n}$ be such that $F(\vec{c})=\vec{a}$ and $G(\vec{d})=%
\vec{b}$. Let $t(\vec{x})$ be an $\mathcal{L}$-term such that $t^{\mathbf{A}%
_{1}}(\vec{c})=f^{\mathbf{A}_{1}}(\vec{c})$ and $t^{\mathbf{B}_{1}}(\vec{d}%
)=f^{\mathbf{B}_{1}}(\vec{d})$. Thus, we have that%
\begin{eqnarray*}
t^{\mathbf{A}}(\vec{a}) &=&t^{\mathbf{A}}(F(\vec{c})) \\
&=&F(t^{\mathbf{A}_{1}}(\vec{c})) \\
&=&F(f^{\mathbf{A}_{1}}(\vec{c})) \\
&=&f^{\mathbf{A}}(F(\vec{c})) \\
&=&f^{\mathbf{A}}(\vec{a})\text{,}
\end{eqnarray*}%
and similarly, $t^{\mathbf{B}}(\vec{b})=f^{\mathbf{B}}(\vec{b})$.

Next we prove that

\begin{enumerate}
\item[(ii)] $\func{Con}\mathbf{A}=\func{Con}\mathbf{A}_{\mathcal{L}}$, for
every $\mathbf{A}\in \mathbb{HSP}_{u}(\mathcal{K})$.
\end{enumerate}

\noindent Let $\mathbf{A}\in \mathbb{HSP}_{u}(\mathcal{K})$ and $\theta \in 
\func{Con}\mathbf{A}_{\mathcal{L}}$. We show that $\theta $ is compatible
with $f$. Suppose $\vec{a},\vec{b}\in A^{n}$ are such that $\vec{a}\ \theta
\ \vec{b}$. By (i) we have an $\mathcal{L}$-term $t\left( \vec{x}\right) $
such that $t^{\mathbf{A}}(\vec{a})=f^{\mathbf{A}}(\vec{a})$ and $t^{\mathbf{A%
}}(\vec{b})=f^{\mathbf{A}}(\vec{b})$. Clearly%
\begin{equation*}
f^{\mathbf{A}}(\vec{a})=t^{\mathbf{A}}(\vec{a})\ \theta \ t^{\mathbf{A}}(%
\vec{b})=f^{\mathbf{A}}(\vec{b})\text{.}
\end{equation*}

Now we shall see that

\begin{enumerate}
\item[(iii)] $\mathbb{HSP}_{u}(\mathcal{K}_{\mathcal{L}})\subseteq \left( 
\mathbb{HSP}_{u}(\mathcal{K})\right) _{\mathcal{L}}$.
\end{enumerate}

\noindent It is always the case that $\mathbb{P}_{u}(\mathcal{K}_{\mathcal{L}%
})=\mathbb{P}_{u}(\mathcal{K})_{\mathcal{L}}$, and (i) implies that $%
\mathcal{L}$-subreducts of algebras in $\mathbb{P}_{u}(\mathcal{K})$ are
closed under $f$. Thus $\mathbb{SP}_{u}(\mathcal{K}_{\mathcal{L}})\subseteq (%
\mathbb{SP}_{u}(\mathcal{K}))_{\mathcal{L}}$, and it only remains to see
that $\mathbb{H}(\mathbb{SP}_{u}(\mathcal{K})_{\mathcal{L}})\subseteq \left( 
\mathbb{HSP}_{u}(\mathcal{K})\right) _{\mathcal{L}}$, which is immediate by
(ii).

By J\'{o}nsson's lemma we have that $\mathbb{V}(\mathcal{K}_{\mathcal{L}%
})_{FSI}\subseteq \mathbb{HSP}_{u}(\mathcal{K}_{\mathcal{L}})$, and so (iii)
produces

\begin{enumerate}
\item[(iv)] $\mathbb{V}(\mathcal{K}_{\mathcal{L}})_{FSI}\subseteq (\mathbb{%
HSP}_{u}(\mathcal{K}))_{\mathcal{L}}$.
\end{enumerate}

Using that $\mathbb{P}_{u}\mathbb{HSP}_{u}(\mathcal{K})\subseteq \mathbb{HSP}%
_{u}(\mathcal{K})$ and (i), it is easy to check that $f$ and $\mathbb{HSP}%
_{u}(\mathcal{K})$ satisfy the conditions stated in (3) of Theorem \ref{term
valued with positive cases}. Thus we can conclude that

\begin{enumerate}
\item[(v)] $f=\left. t_{1}\right\vert _{\varphi _{1}}\cup \ldots \cup \left.
t_{k}\right\vert _{\varphi _{k}}$ in $\mathbb{HSP}_{u}(\mathcal{K})$, with
each $t_{i}$ an $\mathcal{L}$-term and each $\varphi _{i}$ in $\left[
\bigwedge \mathrm{At}(\mathcal{L})\right] $.
\end{enumerate}

\noindent Hence,

\begin{enumerate}
\item[(vi)] the formula $\varphi (\vec{x},z)=(\varphi _{1}(\vec{x})\wedge
z=t_{1}(\vec{x}))\vee \dots \vee (\varphi _{k}(\vec{x})\wedge z=t_{k}(\vec{x}%
))$ defines $f$ in $\mathbb{HSP}_{u}(\mathcal{K})$.
\end{enumerate}

\noindent Observe that (v) implies $\mathbb{HSP}_{u}(\mathcal{K})\vDash
\forall \vec{x}\exists !z\ \varphi (\vec{x},z)$, and by (iv) we obtain $%
\mathbb{V}(\mathcal{K}_{\mathcal{L}})_{FSI}\vDash \forall \vec{x}\exists !z\
\varphi (\vec{x},z)$. Since $\varphi \in \left[ \bigvee \bigwedge \mathrm{At}%
(\mathcal{L})\right] $, Proposition \ref{traduccion} says that there is a
formula $\psi \in \left[ \bigwedge \mathrm{At}(\mathcal{L})\right] $ such
that $\mathbb{V}(\mathcal{K}_{\mathcal{L}})_{FSI}\vDash \varphi
\leftrightarrow \psi $. Thus we have that $\mathbb{V}(\mathcal{K}_{\mathcal{L%
}})_{FSI}\vDash \forall \vec{x}\exists !z\ \psi (\vec{x},z)$ and by Lemma %
\ref{existencia de terminos en variedades aritmeticas} there is an $\mathcal{%
L}$-term $t(\vec{x})$ such that $\mathbb{V}(\mathcal{K}_{\mathcal{L}})\vDash
\forall \vec{x}\ \psi (\vec{x},t(\vec{x}))$. In particular,

\begin{enumerate}
\item[(vii)] $\mathbb{V}(\mathcal{K}_{\mathcal{L}})_{FSI}\vDash \forall \vec{%
x}\ \varphi (\vec{x},t(\vec{x}))$.
\end{enumerate}

Now, if we take $\mathbf{A}\in \mathbb{V}(\mathcal{K})_{FSI}$, by J\'{o}%
nsson's lemma $\mathbf{A}\in \mathbb{HSP}_{u}(\mathcal{K})$, and so by (ii) $%
\mathbf{A}$ and $\mathbf{A}_{\mathcal{L}}$ have the same congruences. Hence $%
\mathbf{A}_{\mathcal{L}}\in \mathbb{V}(\mathcal{K}_{\mathcal{L}})_{FSI}$,
and by (vii) $\mathbf{A}\vDash \forall \vec{x}\ \varphi (\vec{x},t(\vec{x}))$%
. Thus,

\begin{enumerate}
\item[(viii)] $\mathbb{V}(\mathcal{K})_{FSI}\vDash \forall \vec{x}\ \varphi (%
\vec{x},t(\vec{x}))$.
\end{enumerate}

Finally, as (vi) says that $\varphi (\vec{x},z)$ defines $f$ in $\mathbb{HSP}%
_{u}(\mathcal{K})$, it follows from (viii) that $\mathbb{V}(\mathcal{K}%
)_{FSI}\vDash \forall \vec{x}\ t(\vec{x})=f(\vec{x})$. So this identity
holds in $\mathbb{V}(\mathcal{K})$ and the theorem is proved.
\end{proof}

\begin{corollary}
Let $\mathcal{V}$ be an arithmetical variety such that $\mathcal{V}_{FSI}$
is universal. Let $\mathcal{K}\subseteq \mathcal{V}$ be a first order class
and suppose that $\psi (\vec{x},z)$ is a first order formula (in the
language of $\mathcal{V}$) which defines on each algebra $\mathbf{A}$ of $%
\mathcal{K}$ a function $f^{\mathbf{A}}:A^{n}\rightarrow A$. Assume that for
all $\mathbf{A},\mathbf{B}\in \mathcal{K}$ every subalgebra $\mathbf{S}\leq 
\mathbf{A}\times \mathbf{B}$ is closed under $f^{\mathbf{A}}\times f^{%
\mathbf{B}}$. Then $f$ is representable by a term in $\mathcal{K}$.
\end{corollary}

We believe it likely that the Baker-Pixley Theorem holds in scenarios other
than the two considered here (locally finite and arithmetical). A question
we were unable to answer is the following.

Let $\mathcal{K}$ be a first order axiomatizable class of $(\mathcal{L}\cup
\{f\})$-algebras with a majority $\mathcal{L}$-term and suppose that for any 
$\mathbf{A},\mathbf{B}\in \mathcal{K}$ every subalgebra of $\mathbf{A}_{%
\mathcal{L}}\times \mathbf{B}_{\mathcal{L}}$ is closed under $f^{\mathbf{A}%
\times \mathbf{B}}$. Is $f$ representable by an $\mathcal{L}$-term in $%
\mathcal{K}$?


\begin{thebibliography}{99}
\bibitem{ba-dw} R. Balbes and P. Dwinger, Distributive Lattices, University
of Missouri Press, Columbia, Missouri (1974).

\bibitem{ba-pi} K. Baker and F. Pixley, \textit{Polynomial interpolation and
the Chinese Remainder Theorem for algebraic systems}, Math. Z. \textbf{143}
(1975), 165--174.

\bibitem{ba-be} J. Baldwin and J. Berman, \textit{The number of subdirectly
irreducible algebras in a variety}, Algebra Universalis \textbf{5} (1975),
379--389.

\bibitem{bl-pi} W. Blok and D. Pigozzi, \textit{On the congruence extension
property}, Algebra Universalis \textbf{38} (1997), 391--994.

\bibitem{bu} S. Burris, \textit{Remarks on the Fraser-Horn property},
Algebra Universalis \textbf{23} (1986), 19--21.

\bibitem{bu-sa} S. Burris and H. Sankappanavar, A course in Universal Algebra%
\textit{,} Springer-Verlag, New York, 1981.

\bibitem{Ca-Va} M. Campercholi and D. Vaggione, \textit{Algebraically
expandable classes}, Algebra Universalis \textbf{61} (2009), 151--186.

\bibitem{ca-va0} M. Campercholi and D. Vaggione, \textit{Implicit definition
of the quaternary discriminator}, Algebra Universalis \textbf{68} (2012),
1--16.

\bibitem{cz-dz} J. Czelakowski and W. Dziobiak, \textit{Congruence
distributive quasivarieties whose finitely subdirectly irreducible members
form a universal class}, Algebra Universalis \textbf{27} (1990), 128--149.

\bibitem{ge} D. Geiger, \textit{Closed Systems of Functions and Predicates},
Pacific Journal of Mathematics \textbf{27} (1968), 95--100.

\bibitem{gr-va} H. Gramaglia and D. Vaggione, \textit{Birkhoff-like sheaf
representation for varieties of lattice expansions, }Studia Logica \textbf{56%
}(1/2) (1996), 111--131.

\bibitem{kr-cl} P. Krauss and D. Clark, Global subdirect products, Amer.
Math. Soc. Mem. \textbf{210} (1979).

\bibitem{kr} M. Krasner, \textit{Endoth\'{e}orie de Galois abstraite.} S\'{e}%
minaire P. Dubreil (Alg\'{e}bre et Th\'{e}orie des Nombres), \textbf{1}(6),
1968.

\bibitem{Pixley} A. Pixley, \textit{The ternary discriminator function in
universal algebra}, Math. Ann. \textbf{191 }(1971), 167-180.

\bibitem{vo} H. Volger, \textit{Preservation theorems for limits of
structures and global sections of sheaves of structures}, Math. Z. \textbf{%
166} (1970), 27-53.

\bibitem{we} H. Werner, Discriminator algebras, algebraic representation and
model theoretic properties, Akademie-Verlag, Berlin, 1978.
\end{thebibliography}
\end{document}